\documentclass{amsart}
\usepackage{amsfonts,amscd,amsthm,amsgen,amsmath,amssymb}
\usepackage[all]{xy}
\usepackage{epsfig,color}
\usepackage[vcentermath]{youngtab}
\usepackage{tikz}
\usetikzlibrary{decorations.markings,snakes}
\newtheorem{theorem}{Theorem}[section]
\newtheorem{lemma}[theorem]{Lemma}
\newtheorem{corollary}[theorem]{Corollary}
\newtheorem{proposition}[theorem]{Proposition}

\theoremstyle{definition}

\newtheorem{example}[theorem]{Example}

\theoremstyle{remark}
\newtheorem{remark}[theorem]{Remark}

\numberwithin{equation}{section}

\begin{document}

\title[Obstructions to smooth group actions on $4$-manifolds]{Obstructions to smooth group actions on $4$-manifolds from families Seiberg-Witten theory}
\author{David Baraglia}

\address{School of Mathematical Sciences, The University of Adelaide, Adelaide SA 5005, Australia}

\email{david.baraglia@adelaide.edu.au}

\begin{abstract}
Let $X$ be a smooth, compact, oriented $4$-manifold. Building upon work of Li-Liu, Ruberman, Nakamura and Konno, we consider a families version of Seiberg-Witten theory and obtain obstructions to the existence of certain group actions on $X$ by diffeomorphisms. The obstructions show that certain group actions on $H^2(X , \mathbb{Z})$ preserving the intersection form can not be lifted an action of the same group on $X$ by diffeomorphisms. Using our obstructions, we construct numerous examples of group actions which can be realised continuously but can not be realised smoothly for any differentiable structure. For example, we construct compact simply-connected $4$-manifolds $X$ and involutions $f : H^2(X , \mathbb{Z}) \to H^2(X , \mathbb{Z})$ such that $f$ can be realised by a continuous involution on $X$ or by a diffeomorphism, but not by an involutive diffeomorphism for any smooth structure on $X$.
\end{abstract}


\date{\today}



\maketitle


\section{Introduction}

The Nielsen realisation problem, solved by Kerckhoff \cite{ker}, shows that any finite subgroup $G$ of the mapping class group $\pi_0 (Diff(M))$ of a compact oriented surface $M$ of negative Euler characteristic can be realised by an action of $G$ on $M$ by diffeomorphisms. For an infinite group $G$, the Miller-Morita-Mumford classes show that in general there are obstructions to lifting a homomorphism $G \to \pi_0 (Diff(M))$ to $Diff(M)$ \cite{mor}. In this paper, we obtain obstructions to the analogous lifting problem for homomorphisms $G \to \pi_0( Diff^+(X))$, where $X$ is a compact oriented smooth $4$-manifold and $Diff^+(X)$ is the group of orientation preserving diffeomorphisms of $X$. The oriented mapping class group $\pi_0( Diff^+(X))$ is unfortunately not very well understood. To get something more manageable, consider diffeomorphisms modulo pseudo-isotopy. Let $Aut( H^2(X , \mathbb{Z}))$ denote the automorphisms of $H^2(X , \mathbb{Z})$ preserving the intersection pairing. The induced action of $Diff^+(X)$ on $H^2(X , \mathbb{Z})$ descends through pseudo-isotopy giving a homomorphism $\phi : \tilde{\pi}_0( Diff^+(X)) \to Aut( H^2(X , \mathbb{Z}))$, where $\tilde{\pi}_0 (Diff^+(X))$ denotes the group of pseudo-isotopy classes of orientation preserving diffeomorphisms of $X$. When $X$ is simply-connected, it has been shown by Kreck that the map $\phi$ is injective \cite{kre}. This motivates us to consider the following realisation problem: {\em let $X$ be a compact oriented smooth $4$-manifold, let $G$ be a finitely generated group and $\rho : G \to Aut( H^2(X , \mathbb{Z}))$ a homomorphism. Can $\rho$ be realised by an action of $G$ on $X$ by diffeomorphisms?}\\

We note that the map $Diff^+(X) \to Aut( H^2(X , \mathbb{Z}))$ is in general not surjective, however there are many $X$ for which this is the case \cite{wall}. One can also consider the analogous realisation problem for homeomorphisms. To give contrast, let us point out that if $X$ is simply-connected then the map $\pi_0( Homeo^+(X) ) \to Aut( H^2(X , \mathbb{Z}) )$ is known to be surjective by Freedman \cite{fre} and injective by Quinn \cite{qui}. Here $\pi_0( Homeo^+(X) )$ denotes the group of isotopy classes of orientation preserving homeomorphisms of $X$.\\

The simplest type of realisation problem is to take $G = \mathbb{Z}$. A homomorphism $\rho : \mathbb{Z} \to Aut( H^2( X , \mathbb{Z} ) )$ is determined by its value $f = \rho(1) \in Aut( H^2( X , \mathbb{Z} ) )$ on a generator and the realisation problem just asks whether $f$ can be realised by a diffeomorphism. Obstructions to realising elements of $Aut( H^2(X , \mathbb{Z}))$ by diffeomorphisms have been obtained from Donaldson theory (eg, \cite{frmo}) and Seiberg-Witten theory (eg, \cite{lon,na1,lili}). In this paper we will use Seiberg-Witten theory to obtain obstructions for a variety of different groups. Specifically, we will consider cyclic groups of finite even order $\mathbb{Z}_{2k}$, free abelian groups $\mathbb{Z}^d$ and the group $\mathbb{Z}_2 \times \mathbb{Z}_2$. However, our methods are more general and can be applied to many other groups.\\

Our obstructions are obtained by considering a parametrised version of Seiberg-Witten theory for smooth families of $4$-manifolds. The idea of studying Seiberg-Witten theory in families has been used by a number of authors, eg, \cite{liliu,na1,na2,na3,rub,konno,konno2}. In particular, Nakamura \cite{na1,na2} has used families Seiberg-Witten theory to obtain obstructions in the cases $G = \mathbb{Z}$ and $G = \mathbb{Z}^2$. Our obstructions are obtained by similar methods to Nakamura, but are considerably more general. In particular, we are able to relax the condition that the fibrewise tangent bundle of our family of $4$-manifolds admits a Spin$^c$-structure, expanding the range of possible applications.\\

A brief outline our method to obtain obstructions is as follows. Let $\rho : G \to Aut( H^2(X , \mathbb{Z}))$ be given and suppose $\rho$ can be obtained by an action of $G$ on $X$ by diffeomorphisms. Assume that $b^+(X) > 0$ and that we can find a compact manifold $\tilde{B}$ of dimension $d = b^+(X)$ on which $G$ acts freely. We obtain a family $\mathbb{X} \to B$ of $4$-manifolds over $B = \tilde{B}/G$ by taking $\mathbb{X} = X \times_G \tilde{B}$. We would like to think of $B$ as a finite dimensional substitute for the classifying space $BG$ and $\mathbb{X}$ as a finite dimensional substitute for $X \times_G EG$. We equip $\mathbb{X}$ with a family of metrics $\{ g_b \}_{b \in B}$ and generic self-dual $2$-forms $\{ \mu_b \}$ and consider the family of $\mu$-perturbed Seiberg-Witten equations parametrised by $B$. Let $H^+(\mathbb{X}/B) \to B$ be the vector bundle over $B$ whose fibres over $b$ is the space of harmonic self-dual $2$-forms on the corresponding fibres of $\mathbb{X} \to B$. By counting the number of reducible solutions to the Seiberg-Witten equations in the family parametrised by $B$, we are able to deduce, under certain conditions, that the $d$-th Stiefel-Whitney class of $H^+(\mathbb{X}/B)$ vanishes. This is stated as Theorem \ref{t:main} below. On the other hand, we can compute the $d$-th Stiefel-Whitney class of $H^+(\mathbb{X}/B)$ in terms of the homomorphism $\rho : G \to Aut( H^2(X , \mathbb{Z}))$. This gives an obstruction to realising $\rho$ by an action of $X$ by diffeomorphisms.\\

Our main theorem, which is used to produce obstructions to smooth group actions, applies more generally to smooth families of $4$-manifolds. To state the theorem, let $X$ be a compact, oriented, smooth $4$-manifold with $b_1(X) = 0$ and $b^+(X) > 0$. Suppose that $\pi : \mathbb{X} \to B$ is a smooth fibrewise oriented family with fibres diffeomorphic to $X$. Let $\mathcal{S}(X)$ denote the set of Spin$^c$-structures on $X$. Then as explained in Section \ref{sec:fspinc}, the family $\mathbb{X} \to B$ determines a monodromy action of $\pi_1(B)$ on $\mathcal{S}(X)$. Now we are ready to state the theorem:

\begin{theorem}\label{t:main}
Let $X$ be a compact, oriented, smooth $4$-manifold with $b_1(X) = 0$ and $b^+(X) > 0$. Suppose that $\pi : \mathbb{X} \to B$ is a fibrewise oriented family over a compact base of dimension $d = b^+(X)$. Let $\Gamma$ be a monodromy invariant Spin$^c$-structure on $X$ such that $c(\Gamma)^2 > \sigma(X)$. Suppose that one of the following holds:
\begin{itemize}
\item[(1)]{we have $c^2(\Gamma) - \sigma(X) = 8 \; ({\rm mod} \; 16 )$, or}
\item[(2)]{we have that $\Gamma$ extends to a Spin$^c$-structure on $T(\mathbb{X}/B)$ (we will see this holds for instance if $H^3(B , \mathbb{Z}) = 0$).}
\end{itemize}
Then the $d$-th Stiefel-Whitney class of $H^+(\mathbb{X}/B) )$ vanishes.
\end{theorem}

At first glance, this theorem may appear to impose only a very mild constraint on smooth families of $4$-manifolds. However, we show in subsequent sections of the paper that many interesting obstructions to smooth group actions can be obtained through this result. In Section \ref{sec:appl} we consider a number of specific applications and examples. Below we highlight a few such applications. In what follows, $E_8$ denotes the unique compact simply-connected topological $4$-manifold whose intersection form is the $E_8$ lattice.

\begin{proposition}
Let $X$ be the topological $4$-manifold $X = \# a (S^2 \times S^2) \# 2b(-E_8)$ where $a > 3b$ and $b\ge 1$. Then $H^2(X , \mathbb{Z})$ admits an involution $f : H^2(X , \mathbb{Z}) \to H^2(X , \mathbb{Z})$ with the following properties:
\begin{itemize}
\item[(i)]{$f$ can be realised by the induced action of a continuous, locally linear involution $X \to X$.}
\item[(ii)]{$f$ can be realised by the induced action of a diffeomorphism $X \to X$, where the smooth structure is obtained by viewing $X$ as $\# (a-3b)(S^2 \times S^2) \# b(K3)$.}
\item[(iii)]{$f$ can not be realised by the induced action of an involutive diffeomorphism $X \to X$ for any smooth structure on $X$.}
\end{itemize}
\end{proposition}

\begin{proposition}
For any odd integer $k \ge 1$, let $X$ be the topological $4$-manifold $X = \# a (S^2 \times S^2) \# 2kb(-E_8)$ where $a > 3kb$, $b\ge 1$ and $a$ is odd. Then $H^2(X , \mathbb{Z}) = a H \oplus 2kb(-E_8)$, where $H = \left( \begin{matrix} 0 & 1 \\ 1 & 0 \end{matrix} \right)$. Let $f : H^2(X , \mathbb{Z}) \to H^2(X , \mathbb{Z})$ be an isometry of order $2k$ which acts as $\left( \begin{matrix} 0 & -1 \\ -1 & 0 \end{matrix} \right)$ on each $H$ summand and acts as a permutation of the $-E_8$ summands such that each cycle of the permutation has length $2k$. Then $f$ has the following properties:
\begin{itemize}
\item[(i)]{$f$ can be realised by a continuous, locally linear action of $\mathbb{Z}_{2k}$ on $X$.}
\item[(ii)]{$f$ can be realised by the induced action of a diffeomorphism $X \to X$, where the smooth structure is obtained by viewing $X$ as $\# (a-3kb)(S^2 \times S^2) \# bk(K3)$.}
\item[(iii)]{$f$ can not be realised by a smooth $\mathbb{Z}_{2k}$-action for any smooth structure on $X$.}
\end{itemize}
\end{proposition}

\begin{proposition}
Let $X$ be the topological $4$-manifold $X = \# 2 \mathbb{CP}^2 \# 11 \overline{\mathbb{CP}^2}$. There exists a commuting pair of isometries $f_1 , f_2 : H^2(X , \mathbb{Z}) \to H^2(X , \mathbb{Z})$ such that:
\begin{itemize}
\item[(i)]{$f_1$ and $f_2$ can be realised as diffeomorphisms of $X$ with respect to its standard smooth structure.}
\item[(ii)]{For any smooth structure on $X$, any diffeomorphisms realising $f_1$ and $f_2$ do not commute.}
\end{itemize}
\end{proposition}

\begin{proposition}
Let $X$ be the topological $4$-manifold $X = \# 2(a+b) \mathbb{CP}^2 \# (2a+2b+32c+1) \overline{\mathbb{CP}^2}$, where $a,b,c$ are positive integers such that $a,b \ge 3c$. There exists a commuting pair of involutive isometries $\phi_1 , \phi_2 : H^2(X , \mathbb{Z}) \to H^2(X , \mathbb{Z})$ with the following properties:
\begin{itemize}
\item[(i)]{$\phi_1,\phi_2$, can be realised by a continuous, locally linear $\mathbb{Z}_2 \times \mathbb{Z}_2$-action on $X$.}
\item[(ii)]{Viewing $X$ as the smooth $4$-manifold $X = \# (2a+2b-6c) S^2 \times S^2 \# 2c (K3) \# \overline{\mathbb{CP}^2}$, we have that $\phi_1$ and $\phi_2$ can be realised as smooth involutions on $X$.}
\item[(iii)]{For any smooth structure on $X$, we have that $\phi_1$ and $\phi_2$ can not be realised as commuting smooth involutions.}
\end{itemize}
\end{proposition}

A brief summary of the contents of the paper is as follows. In Section \ref{sec:fsw} we consider the Seiberg-Witten equations for families of $4$-manifolds. We consider in particular in  \ref{sec:fspinc} the obstructions to finding a Spin$^c$-structure on the vertical tangent bundle of a family of $4$-manifolds compatible with a given Spin$^c$-structure on the $4$-manifold. In \ref{sec:fmoduli} we show how the families moduli space can be constructed even in cases where no such Spin$^c$-structure exists. In \ref{sec:freducibles} we study the local structure of the families moduli space around reducible solutions, by computing the tangent and obstruction spaces for the Kuranishi model. In \ref{sec:vanishing} we prove Theorem \ref{t:hplusobstruction}, which serves as the basis for our obstructions to smooth group actions on $4$-manifolds. In Sections \ref{sec:z2}-\ref{sec:z2z2}, we consider obstructions for the following specific classes of groups, namely $\mathbb{Z}_2$ (\ref{sec:z2}), cyclic groups of even order (\ref{sec:z2m}), free abelian groups (\ref{sec:freeab}) and $\mathbb{Z}_2 \times \mathbb{Z}_2$ (\ref{sec:z2z2}). Note that the applications of Theorem \ref{t:hplusobstruction} are certainly not limited to just these groups. Finally in Section \ref{sec:appl} we consider a number of specific applications and examples of our obstruction results obtained in the previous sections.\\

\noindent{\bf Acknowledgments}. I would like to thank Diarmuid Crowley, Hokuto Konno and Mathai Varghese for many helpful conversations. In particular I thank Hokuto Konno for pointing out the definition of the group $\mathcal{G}_0$ in Section \ref{sec:fmoduli}. D. Baraglia is financially supported by the Australian Research Council Discovery Early Career Researcher Award DE160100024 and Australian Research Council Discovery Project DP170101054.

\section{Families Seiberg-Witten moduli spaces}\label{sec:fsw}

Let $X$ be a compact, oriented, smooth $4$-manifold. Let $\pi : \mathbb{X} \to B$ be a smooth, locally trivial fibre bundle with fibres diffeomorphic to $X$ and with base a smooth compact manifold $B$ of dimension $d$. For $b \in B$, we write $X_b$ for the fibre of $\mathbb{X}$ over $b$. We wish to study the Seiberg-Witten equations for the family of $4$-manifolds $\mathbb{X}$ parametrised by $B$. To define the ordinary Seiberg-Witten equations on $X$ requires a metric, a self-dual $2$-form perturbation and a Spin$^c$-structure on $X$. In the families setting we require parametrised versions of these. Let $T(\mathbb{X}/B) = Ker(\pi_*)$ denote the vertical tangent bundle, whose restriction to $X_b$ is the usual tangent bundle. A fibrewise metric for the family $\mathbb{X}/B$ is a metric $g$ on $T(\mathbb{X}/B)$. Similarly a fibrewise self-dual $2$-form for $\mathbb{X}/B$ is a section $\mu$ of $\wedge^2_+ T(\mathbb{X}/B)^* $. Note that in order to define the bundle $\wedge^2_+ T(\mathbb{X}/B)^* $ of fibrewise self-dual $2$-forms it is necessary to assume that $\mathbb{X}/B$ is fibrewise oriented. Henceforth we will always assume that our family $\mathbb{X}/B$ is equipped with a fibrewise orientation. However the base $B$ is not assumed to be orientable. A fibrewise Spin$^c$-structure for the family $\mathbb{X}/B$ is by definition a Spin$^c$-structure on $T(\mathbb{X}/B)$. By considering the notion of twisted Spin$^c$-structures, we will see that it is still possible to construct families Seiberg-Witten moduli spaces even when the bundle $T(\mathbb{X}/B)$ is not Spin$^c$.

\subsection{Spin$^c$-structures in families}\label{sec:fspinc}

Let $\mathcal{S}(X)$ denote the set of Spin$^c$-structures on $X$ compatible with the chosen orientation. A priori this set depends on the choice of metric, but since the space of metrics on $X$ is contractible there is a uniquely determined identification between the Spin$^c$-structures associated to any pair of metrics. As is well known, any smooth oriented $4$-manifold has a Spin$^c$-structure and the set $\mathcal{S}(X)$ is in a natural way a torsor for the group of complex line bundles on $X$, which is isomorphic to $H^2(X , \mathbb{Z})$. We denote the action of a line bundle $L$ on $\mathcal{S}(X)$ by $\Gamma \mapsto \Gamma \otimes L$, where $\Gamma \in \mathcal{S}(X)$. The orientation preserving diffeomorphisms of $X$ act on $\mathcal{S}(X)$ via pullback (and the canonical identification of Spin$^c$-structures for different metrics). Clearly this pullback action is compatible with the $H^2(X , \mathbb{Z})$-torsor structure in the sense that if $f : X \to X$ is a diffeomorphism, $\Gamma \in \mathcal{S}(X)$ and $L \in H^2(X , \mathbb{Z})$, then $f^*( \Gamma \otimes L) = f^*(\Gamma) \otimes f^*(L)$. More generally, we say that a bijection $\phi : \mathcal{S}(X) \to \mathcal{S}(X)$ is an {\em affine} transformation if there exists a group automorphism $\psi : H^2(X , \mathbb{Z}) \to H^2(X , \mathbb{Z})$ such that for all $\Gamma \in \mathcal{S}(X)$ and $L \in H^2(X , \mathbb{Z})$, we have
\[
\phi( \Gamma \otimes L ) = \phi(\Gamma) \otimes \psi(L).
\]
We call $\psi$ the linear part of $\phi$. In particular, a diffeomorphism $f : X \to X$ acts on $\mathcal{S}(X)$ as an affine transformation with linear part given by the usual pullback action of $f$ on $H^2(X , \mathbb{Z})$. Observe also that the action of $f$ on $\mathcal{S}(X)$ depends only on $f$ up to isotopy.\\

Consider again the family $\pi : \mathbb{X} \to B$. The collection of Spin$^c$-structures $\{ \mathcal{S}(X_b) \}_{b \in B}$ for each $b \in B$ forms a flat bundle over $B$, which we denote by $\mathcal{S}(\mathbb{X}/B) \to B$. The monodromy of this flat bundle is given by a representation $\rho : \pi_1(B) \to Aff( \mathcal{S}(X) )$ of $\pi_1(B)$ on $\mathcal{S}(X)$ acting by affine transformations (to obtain the monodromy representation, one chooses a basepoint $0 \in B$ and a diffeomorphism $X_0 \cong X$). Similarly, for any abelian group $A$, the collection of cohomology groups $\{ H^i(X_b , A ) \}_{b \in B}$ forms a flat bundle over $B$, namely the local system $R^i \pi_* A$. We  observe that that linear part of $\rho$ is the monodromy representation $\pi_1(B) \to Aut( H^2( X , \mathbb{Z} ) )$ corresponding to the local system $R^2 \pi_* \mathbb{Z}$.\\

Suppose that $\widetilde{\Gamma}$ is a families Spin$^c$-structure on $\mathbb{X}/B$, that is, suppose $\widetilde{\Gamma}$ is a Spin$^c$-structure on $T(\mathbb{X}/B)$. The restriction $\Gamma_b = \widetilde{\Gamma}|_{X_b}$ of $\widetilde{\Gamma}$ to each fibre $X_b$ determines a global section of the bundle $\mathcal{S}(\mathbb{X}/B)$, which in turn corresponds to a monodromy invariant Spin$^c$-structure $\Gamma \in \mathcal{S}(X)$. Thus a necessary condition for the existence of a families $Spin^c$-structure is that there exists a monodromy invariant Spin$^c$-structure on $X$. The following proposition shows that in certain cases this is also a sufficient condition:

\begin{proposition}\label{p:famspinc}
Let $\Gamma \in \mathcal{S}(X)$ be a monodromy invariant Spin$^c$-structure on $X$. Suppose that $b_1(X) = 0$ and $H^3( B , \mathbb{Z} )=0$. Then there exists a families Spin$^c$-structure $\widetilde{\Gamma}$ on $\mathbb{X}/B$ whose restriction to $X_0 \cong X$ is $\Gamma$. If $\Gamma_1 , \Gamma_2$ are two families Spin$^c$-structures on $\mathbb{X}/B$ whose restrictions to $X_0$ agree, then $\Gamma_2 = \Gamma_1 \otimes \pi^*(M)$ for some line bundle $M \in H^2(B , \mathbb{Z})$.
\end{proposition}
\begin{proof}
First we will show that under the above assumptions, there exists a families Spin$^c$-structure on $\mathbb{X}/B$ which does not necessarily restrict to $\Gamma$ on $X_0$. We will then see that it is possible to change this Spin$^c$-structure to one that does restrict to $\Gamma$ on $X_0$.\\

The obstruction to the existence of a Spin$^c$-structure on $T(\mathbb{X}/B)$ is the third integral Stiefel-Whitney class $W_3( T(\mathbb{X}/B) ) \in H^3( \mathbb{X} , \mathbb{Z} )$. Consider the $E_2$-page of the Leray-Serre spectral sequence for the cohomology of $\mathbb{X}$ with integral coefficients: $E_2^{p,q} = H^p( B , R^q \pi_* \mathbb{Z})$. The restriction of $W_3( T(\mathbb{X}/B) )$ to the fibres gives a class in $E_2^{0,3}$, however this class vanishes, since the fibres of $\mathbb{X}/B$ are Spin$^c$ and so their integral third Stiefel-Whitney class is zero. It follows that $W_3( T(\mathbb{X}/B) )$ determines a class in $E_2^{1,2} = H^1( B , R^q \pi_* \mathbb{Z} )$. This class is the obstruction to the existence of a Spin$^c$-structure on the restriction of $\mathbb{X} \to B$ to a $1$-skeleton of $B$. The existence of a monodromy invariant Spin$^c$-structure $\Gamma$ on $X$ certainly allows us to construct a Spin$^c$-structure on the restriction of the family to a $1$-skeleton of $B$ (the monodromy invariance of $\Gamma$ ensures that each time we attach a $1$-cell, $\Gamma$ can be extended over it). Therefore, the class of $W_3( T(\mathbb{X}/B) )$ in $E_2^{1,2}$ vanishes. Next, observe that $b_1(X) = 0$ implies $H^1(X , \mathbb{Z}) = 0$ and thus $E_2^{2,1} = 0$. It follows that $W_3( T(\mathbb{X}/B) )$ is the pullback of a class in $E_2^{3,0} = H^3( B , \mathbb{Z})$. By assumption $H^3(B , \mathbb{Z}) = 0$ and so $W_3( T(\mathbb{X}/B) ) = 0$. Hence there exists a Spin$^c$-structure on $T(\mathbb{X}/B)$.\\

Let $\widetilde{\Gamma}$ be a Spin$^c$-structure on $T(\mathbb{X}/B)$. The fibrewise restriction of $\widetilde{\Gamma}$ determines a monodromy invariant Spin$^c$-structure $\Gamma' \in \mathcal{S}(X)$. Let $L \in H^2(X , \mathbb{Z}))$ be defined by $\Gamma = \Gamma' \otimes L$. From this it follows that $L$ is monodromy invariant as an element of $H^2(X , \mathbb{Z})$. We will show that there exists a line bundle $\widetilde{L}$ on $\mathbb{X}$ whose fibrewise restriction is $L$. In this case, $\widetilde{\Gamma} \otimes \widetilde{L}$ is a Spin$^c$-structure on $T(\mathbb{X}/B)$ whose fibrewise restriction is $\Gamma$, as required. To show the existence of such a line bundle, we again turn to the Leray-Serre spectral sequence $E_2^{p,q}$. The monodromy invariant line bundle $L$ determines a class $L \in E_2^{0,2}$. Let $d_2, d_3, \dots $ denote the differentials of the spectral sequence. The first obstruction to extending this to a line bundle on $\mathbb{X}$ is given by $d_2( L) \in E_2^{2,1}$. However, recall that we have assumed $b_1(X) = 0$ so that $E_2^{2,1} = 0$. Thus $d_2(L) = 0$ and $L$ determines a class $L \in E_3^{0,2}$ on the $E_3$-page. The second and final obstruction to extending $L$ to a line bundle on $\mathbb{X}$ is given by $d_3(L) \in E_3^{3,0}$. Note that since $E_2^{1,1} = 0$ (again, because $b_1(X) = 0$) we have $E_3^{3,0} \cong E_2^{3,0} = H^3(B , \mathbb{Z})$, which is zero by assumption. Thus $L$ extends to a line bundle on $\mathbb{X}$. The final statement of the proposition also follows easily from the Leray-Serre spectral sequence.
\end{proof}

If $\Gamma \in \mathcal{S}(X)$ is a monodromy invariant Spin$^c$-structure, then its characteristic class $c(\Gamma) \in H^2( X , \mathbb{Z})$ is also monodromy invariant. The following proposition gives conditions under which the converse holds:
\begin{proposition}
Let $c \in H^2( X , \mathbb{Z})$ be a monodromy invariant lift of $w_2(TX)$. Let $H^2(X , \mathbb{Z} )_2$ denote the $2$-torsion subgroup of $H^2(X , \mathbb{Z})$ and observe that $\pi_1(B)$ acts on $H^2( X , \mathbb{Z} )_2$ by monodromy. If $H^1( \pi_1(B) , H^2( X , \mathbb{Z} )_2 ) = 0$ (for instance, if $H^2(X , \mathbb{Z})$ has no $2$-torsion) then there exists a monodromy invariant Spin$^c$ structure $\Gamma \in \mathcal{S}(X)$ such that $c = c(\Gamma)$. Any two such Spin$^c$-structures differ by a monodromy invariant line bundle of order $2$.
\end{proposition}
\begin{proof}
Since $c$ is a lift of $w_2(TX)$, there exists a Spin$^c$-structure $\Gamma' \in \mathcal{S}(X)$ such that $c = c(\Gamma')$. In general $\Gamma'$ will fail to be monodromy invariant. However, for each $g \in \pi_1(B)$ we have $g^*( \Gamma') = \Gamma' \otimes L(g)$ for some $L(g) \in H^2(X , \mathbb{Z})$. Since $c$ is monodromy invariant, we have $c(\Gamma') = c(\Gamma' \otimes L(g))$ and thus $L(g)^2 = 1$, the trivial line bundle. Thus $L(g) \in H^2(X , \mathbb{Z})_2$. In this way, we obtain a map $L : \pi_1(B) \to H^2(X , \mathbb{Z})_2$ and it is easily checked that $L$ satisfies the group cocycle condition $L(gh) = L(g) \otimes g^*( L(h) )$. Clearly the cohomology class of this cocycle in $H^1( \pi_1(B) , H^2(X , \mathbb{Z})_2)$ is the obstruction to finding a monodromy invariant lift of $c$ to a Spin$^c$-structure. In particular, if $H^1( \pi_1(B) , H^2(X , \mathbb{Z})_2) = 0$ then a monodromy invariant Spin$^c$-structure $\Gamma \in \mathcal{S}(X)$ with $c = c(\Gamma)$ exists. Let $\Gamma, \hat{\Gamma}$ be two monodromy invariant Spin$^c$-structures with $c = c(\Gamma) = c(\hat{\Gamma})$. Then $\hat{\Gamma} = \Gamma \otimes L$ for some line bundle $L$. Clearly $L$ must be monodromy invariant, since $\Gamma$ and $\hat{\Gamma}$ are. Moreover $L$ has order $2$ since $c(\Gamma) = c(\hat{\Gamma})$.
\end{proof}

\subsection{The families moduli space}\label{sec:fmoduli}

Suppose that we have a family $\mathbb{X}/B$ equipped with fibrewise metric and perturbation, and that we are given a Spin$^c$-structure on $T(\mathbb{X}/B)$. The moduli space of gauge equivalence classes of solutions to the Seiberg-Witten equations for this family with respect to the given data (metrics, perturbations, Spin$^c$-structure) defines a compact Hausdorff topological space $\mathcal{M}$ together with a continuous map $\mathcal{M} \to B$. For generic metrics and perturbations the families moduli space is a smooth manifold away from reducible solutions. We will return to this point later, but for the now we do not make any smoothness assumptions on $\mathcal{M}$. Now suppose that we have a family $\mathbb{X}/B$ equipped with fibrewise metric and perturbation and suppose that $\Gamma \in \mathcal{S}(X)$ is a monodromy invariant Spin$^c$-structure which does not necessarily extend to a Spin$^c$-structure on $T(\mathbb{X}/B)$. Let $\{ U_\alpha \}_{\alpha \in I}$ be an open cover of $B$ such that for each $\alpha \in I$ the restriction $\mathbb{X}_\alpha = \pi^{-1}(U_\alpha)$ of the family to $U_\alpha$ is trivial. The monodromy invariant Spin$^c$-structure $\Gamma$ can be viewed as a section of the bundle $\mathcal{S}(\mathbb{X}/B)$ of Spin$^c$-structures on each fibre and thus determines a Spin$^c$-structure on each fibre. There is no obstruction to extending this to a Spin$^c$-structure on $T(\mathbb{X}_\alpha/B)$, the restriction of the vertical tangent bundle over $U_\alpha$. Thus we may construct a families Seiberg-Witten moduli space $\mathcal{M}_\alpha \to U_\alpha$ over each $U_\alpha$.\\

Consider the restriction of the families $\mathbb{X}_\alpha / B$, $\mathbb{X}_\beta/B$ over $\pi^{-1}(U_\alpha \cap U_\beta)$. Since the locally defined Spin$^c$-structures are obtained from a global section $\Gamma : B \to \mathcal{S}(\mathbb{X}/B)$, it follows that the Spin$^c$-structures on $T(\mathbb{X}_\alpha/B)$ and $T(\mathbb{X}_\beta/B)$ are isomorphic on $U_\alpha \cap U_\beta$. Let $\mathbb{S}_\alpha \to \mathbb{X}_\alpha$ denote the spinor bundle associated to the Spin$^c$-structure on $T(\mathbb{X}_\alpha/B)$. Over $U_\alpha \cap U_\beta$ we obtain an isomorphism of Spin$^c$-structures $\varphi_{\alpha \beta} : \mathbb{S}_\beta \to \mathbb{S}_\alpha$. If the isomorphisms $\{ \varphi_{\alpha \beta} \}$ satisfy the cocycle condition then the locally defined Spin$^c$-bundles can be patched together to form a globally defined Spin$^c$-bundle, giving a Spin$^c$-structure on $T(\mathbb{X}/B)$. In general, the $\varphi_{\alpha \beta}$ only form a cocycle up to an intertwiner of the Clifford action, that is up to a $U(1)$-valued function. Stated differently, let $Diff(X,\Gamma)$ be the group of orientation preserving diffeomorphisms of $X$ which preserve the isomorphism class of $\Gamma$ and let $Aut(X,\Gamma)$ be the group consisting of pairs $(f , \varphi)$, where $f \in Diff(X,\Gamma)$ and $\varphi : f^*(S)  \to S$ is an isomorphism of Spin$^c$-structures. Then we have a short exact sequence \cite{na3,konno2}:
\[
1 \to \mathcal{G} \to Aut(X,\Gamma) \to Diff(X,\Gamma) \to 1.
\]
The family $\mathbb{X} \to B$ is constructed from transition functions $f_{\alpha \beta}$ valued in $Diff(X,\Gamma)$ and the $\varphi_{\alpha \beta}$ are a choice of lifts of $f_{\alpha \beta}$ to  $Aut(X,\Gamma)$. In general the lifts fail to be a cocycle, and thus on the triple overlap $U_\alpha \cap U_\beta \cap U_\gamma$, there exists a $U(1)$-valued function $g_{\alpha \beta \gamma}$ such that $g_{\alpha \beta \gamma} = \varphi_{\alpha \beta} \varphi_{\beta \gamma} \varphi_{\gamma \alpha}$. From its definition we have that $\{ g_{\alpha \beta \gamma} \}$ is a $U(1)$-valued cocycle whose cohomology class in $H^2( \mathbb{X} , U(1) ) \cong H^3( \mathbb{X} , \mathbb{Z} )$ is $W_3( T(\mathbb{X}/B) )$. A more geometric point of view is to say that $\{ g_{\alpha \beta \gamma } \}$ defines a bundle gerbe on $\mathbb{X}$ and that the locally defined spinor bundles $\{ \mathbb{S}_\alpha \}$ gives a bundle gerbe module. In any case, the isomorphisms $\varphi_{\alpha \beta} : \mathbb{S}_\beta |_{U_\alpha \cap U_\beta} \to \mathbb{S}_\alpha |_{U_\alpha \cap U_\beta}$ induce homeomorphisms $\widehat{\varphi}_{\alpha \beta} : \mathcal{M}_\beta |_{U_\alpha \cap U_\beta} \to \mathcal{M}_\alpha |_{U_\alpha \cap U_\beta}$ between the corresponding moduli spaces. Moreover, since $g_{\alpha \beta \gamma} = \varphi_{\alpha \beta} \varphi_{\beta \gamma} \varphi_{\gamma \alpha}$ is given by a $U(1)$-gauge transformation, it follows that the induced homeomorphisms $\widehat{\varphi}_{\alpha \beta}$ satisfy the cocycle condition $\widehat{\varphi}_{\alpha \beta} \widehat{\varphi}_{\beta \gamma} \widehat{\varphi}_{\gamma \alpha} = 1$. Therefore the collection of local families moduli spaces $\{ \mathcal{M}_\alpha \}$ patch together to form a well-defined global families moduli space $\mathcal{M} \to B$. In summary, in order to construct a families moduli space $\mathcal{M} \to B$, it is sufficient to have a monodromy invariant Spin$^c$-structure $\Gamma \in \mathcal{S}(X)$. We do not require that $\Gamma$ extends to a Spin$^c$-structure on $T(\mathbb{X}/B)$.\\

In addition to the families moduli space, it will be necessary to consider a families configuration space and its quotient under the gauge group. First let us recall the unparametrised configuration space. Let $(g , \mu)$ be a pair consisting of a metric $g$ on $X$ and a $g$-self-dual $2$-form $\mu$ and let $\Gamma \in \mathcal{S}(X)$. We define the reduced configuration space $\mathcal{B}^*_{g,\mu}(\Gamma)$ to be the space of all pairs $(A , \varphi)$ where $A$ is a connection on the determinant line associated to $\Gamma$ and $\varphi$ is a section of the positive spinor bundle associated to $\Gamma$ which is not identically zero. The gauge group $\mathcal{G} = Map(X , S^1)$ acts freely on $\mathcal{B}_{g,\mu}^*(\Gamma)$ and we define the reduced configuration space modulo gauge transformations associated to $(g,\mu)$ and $\Gamma$ to be the quotient $\mathcal{C}_{g,\mu}^*(\Gamma) = \mathcal{B}_{g,\mu}^*(\Gamma)/\mathcal{G}$. We recall that if $b_1(X) = 0$, then $\mathcal{C}_{g,\mu}^*(\Gamma)$ has the homotopy type of $\mathbb{CP}^{\infty}$. Recall in this case that there is a naturally defined unitary line bundle $L \to \mathcal{C}^*_{g,\mu}(\Gamma)$ which represents a generator of $H^2( \mathbb{CP}^\infty , \mathbb{Z})$. To see this, let
\begin{equation}\label{equ:defg0}
\mathcal{G}_0 = \left\{ g \in \mathcal{G} \; | \; g(x) = e^{if(x)}, \; \text{for some } f : X \to \mathbb{R} \; \text{such that } \int_X f \, dvol_{X,g} = 0 \right\}.
\end{equation}
Note that since $b_1(X) = 0$, we have that every $g \in \mathcal{G}$ can be written as $g = e^{if}$ for some $f : X \to \mathbb{R}$. It follows that there is a natural split short exact sequence
\[
1 \to \mathcal{G}_0 \to \mathcal{G} \to S^1.
\]
 We now define $L$ as the unitary line bundle associated to the principal circle bundle $\mathcal{B}^*_{g,\mu}(\Gamma)/\mathcal{G}_0 \to \mathcal{B}^*_{g,\mu}(\Gamma)/\mathcal{G} = \mathcal{C}^*_{g,\mu}$.\\

We now consider a families version of the above construction. Thus assume we have a family $\pi : \mathbb{X} \to B$ equipped with a family $(g , \mu)$ of metrics and self-dual $2$-forms. Suppose that $\Gamma \in \mathcal{S}(X)$ is a monodromy invariant Spin$^c$-structure on $X$. First, let us suppose that $\Gamma$ extends to a Spin$^c$-structure on $\mathbb{X}$. Let $\mathbb{S}^+ \to \mathbb{X}$ be the associated positive spinor bundle of this Spin$^c$-structure. This has the property that the restriction of $\mathbb{S}^+$ to any fibre $X_b$ is the positive spinor bundle of $\Gamma_b$. We define the families reduced configuration space $\mathcal{B}^*_B(\Gamma)$ to be the space of all triples $(b , A , \varphi)$, where $b \in B$, $A$ is a connection on the restriction of the determinant line to $X_b$ and $\varphi$ is a section of $\mathbb{S}^+|_{X_b}$ which is not identically zero. The gauge group $\mathcal{G}$ acts freely on $\mathcal{B}^*_B(\Gamma)$ and we define the families reduced configuration space modulo gauge transformations $\mathcal{C}^*_B(\Gamma)$ to be the quotient $\mathcal{C}^*_B(\Gamma) = \mathcal{B}^*_B(\Gamma)/\mathcal{G}$ (strictly speaking, $\mathcal{G}$ is a bundle of groups over $B$ and can be thought of as a groupoid. Then $\mathcal{G}$ acts on $\mathcal{B}^*_B(\Gamma)$ in the sense of groupoid actions).\\

We have that $\mathcal{C}^*_B(\Gamma)$ is a locally trivial fibre bundle over $B$ whose fibre over $b \in B$ is the reduced configuration space modulo gauge transformations $\mathcal{C}^*_{g_b , \mu_b}(\Gamma_b)$. If $b_1(X) = 0$ this is a fibre bundle over $B$ whose fibres have the homotopy type of $\mathbb{CP}^\infty$. In this case, we claim that there is a line bundle $L_B \to \mathcal{C}^*_B(\Gamma)$ whose restriction to each fibre of the projection $\mathcal{C}^*_B(\Gamma) \to B$ represents a generator of $H^2( \mathbb{CP}^\infty, \mathbb{Z})$. In fact, this line bundle is constructed in exactly the same way as in the unparametrised case. To construct $L_B$ we just need to define for each $b \in B$ the subgroup $\mathcal{G}_0(X_b)$ of the gauge group $\mathcal{G}(X_b)$ for the fibre $X_b$. We not that since the fibre bundle $\mathbb{X} \to B$ is fibrewise oriented and equipped with a families metric $g = \{ g_b \}$, we obtain fibrewise volume forms $dvol_{X_b , g_b}$ and therefore we can repeat the definition (\ref{equ:defg0}) in families. Lastly, suppose that $\Gamma \in \mathcal{S}(X)$ is a monodromy invariant Spin$^c$-structure which does not necessarily extend to a Spin$^c$-structure on $T(\mathbb{X}/B)$. Then we can still define spaces $\mathcal{B}^*_{U_\alpha}(\Gamma)$ with respect to some open cover $\{ U_\alpha \}$ of $B$ and we can form their quotient spaces $\mathcal{C}^*_{U_\alpha}(\Gamma)$. Since we have divided out by gauge transformations, we have that the $\{ \mathcal{C}^*_{U_\alpha}(\Gamma) \}$ can be patched together to form a fibre bundle $\mathcal{C}^*_B(\Gamma) \to B$ whose fibre over $b \in B$ can be identified with $\mathcal{C}^*_{g_b , \mu_b}(\Gamma_b)$. However, our construction of the line bundle $L_B$ breaks down. Thus we can only assume $L_B$ exists if $\Gamma$ extends to a Spin$^c$-structure on $T(\mathbb{X}/B)$.

\subsection{Structure of the moduli space around reducibles}\label{sec:freducibles}

The structure of the families Seiberg-Witten moduli space can be analysed in much the same way as the ordinary Seiberg-Witten moduli space, as in \cite{nic,sal}. In proving that the usual Seiberg-Witten invariants are independent of the choice of metric and perturbation, one needs to consider a $1$-parameter families version of the Seiberg-Witten moduli space. The same techniques may be applied to the Seiberg-Witten equations parametrised by a smooth manifold $B$ of any dimension. In particular, the same techniques show that for any choice of families metric $g = \{ g_b \}_{b \in B}$ and for a generic families perturbation $\mu = \{ \mu_b \}_{b \in B}$, the families moduli space $\mathcal{M}(g_t,\mu_t) \to B$ is smooth away from reducible solutions. Here, generic means a subset of Baire second category with respect to the $\mathcal{C}^\infty$-topology.\\

Let $g = \{g_b\}_{b \in B}$ be a fibrewise metric on the family $\mathbb{X}/B$. This determines a rank $b^+(X)$ orthogonal vector bundle $H^+( \mathbb{X}/B)$ whose fibre over $b \in B$ is $H^+(X_b)$, the space of $g_b$-self-dual harmonic $2$-forms on $X_b$. This is a sub-bundle of the flat vector bundle $H^2( \mathbb{X}/B , \mathbb{R} ) = R^2 \pi_* \mathbb{R}$ whose fibre over $b$ is $H^2(X_b , \mathbb{R})$. We note that although $H^+( \mathbb{X}/B)$ depends on the choice of metric $g$ when considered as a sub-bundle of $H^2(\mathbb{X}/B)$, the underlying vector bundle $H^+(\mathbb{X}/B)$ is up to isomorphism independent of the choice of metric. This is because $H^+(\mathbb{X}/B)$ is a maximal positive definite subbundle of $H^2( \mathbb{X}/B , \mathbb{R} )$, but the space of all such maximal positive definite subbundles is contractible.\\

Let $\mu = \{ \mu_b \}_{b\in B}$ be a family of self-dual $2$-forms with respect to $g$, generic in sense described above. The $L^2$-orthogonal projection of $\mu_b$ to $H^+( X_b)$ determines a section $[\mu] : B \to H^+(\mathbb{X}/B)$. Similarly, the locus of perturbations for which the Seiberg-Witten equations admits a reducible solution defines a section $\mathcal{W} : B \to H^+(\mathbb{X}/B)$. We may additionally assume that $\mu$ is chosen sufficiently generically so that it intersects the locus $\mathcal{W}$ transversally. In the case that $b^+(X) > d = dim(B)$, transversality means that $[\mu]$ and $\mathcal{W}$ are disjoint.\\

We will be interested in the case $b^+(X) = d$, where transversality implies that $[\mu]$ and $\mathcal{W}$ intersect in a finite set of points, provided $B$ is compact. The intersection $[\mu] \cap \mathcal{W}$, counted mod $2$ is given by the $d$-th Stiefel Whitney class $w_d( H^+(\mathbb{X}/B) ) \in H^d(B , \mathbb{Z}_2) \cong \mathbb{Z}_2$. One can also consider the case $b^+(X) < d$, in which case if $[\mu]$ intersects $\mathcal{W}$ transversally, then their intersection is given by a smooth submanifold of $B$ of dimension $d - b^+(X)$. This submanifold is mod $2$ Poincar\'e dual to the Stiefel-Whitney class $w_{b^+(X)}( H^+( \mathbb{X}/B ) ) \in H^{b^+(X)}( B , \mathbb{Z}_2 )$.\\

We will need to examine the structure of the families moduli space around a reducible solution. Suppose $b \in B$ is such that $[\mu_b]$ intersects the wall $\mathcal{W}_b$. For simplicity we will consider the case that $b_1(X) = 0$, although this can be relaxed. It follows that there is a unique reducible point $(A,0)$ of the families moduli space lying over $b$. The reducible point is given by a connection $A$ on the determinant line bundle of $S_b$ and satisfying $F^+_A + i\mu_b = 0$. We would like to describe the Zariski tangent space and obstruction space at $(A,0)$ in the families moduli space. First consider the deformation complex for the unparameterised Seiberg-Witten equations on $X_b$ around the reducible point $(A, 0)$. We work with configurations of class $L^{k+1,p}$ for some integer $k \ge 0$ and real number $p > 1$. Then the deformation complex takes the form:
\[
0 \to L^{k+2,p}( X_b , i\mathbb{R}) \buildrel \partial_1 \over \longrightarrow L^{k+1,p}(X , iT^*X_b \oplus S^+_b ) \buildrel \partial_2 \over \longrightarrow L^{k,p}(X_b , i \wedge^2_+ T^*X_b \oplus S^-_b) \to 0,
\]
where $\partial_1( f ) = ( 2df  , 0 )$, $\partial_2( a , \psi ) = ( d^+a , D_A \psi )$. For $i=0,1,2$, let $H^i( A,0 )$ denote the cohomology groups of this complex. In particular, $H^1(A,0) \cong Ker(D_A)$ is the Zariski tangent space at $(A,0)$ in the unparametrised moduli space and $H^2(A,0) \cong H^+(X_b) \oplus Coker(D_A)$ is the obstruction space at $(A,0)$ in the unparametrised moduli space.\\

Next, we consider the deformation complex for the families moduli space around $(A,0)$. This requires some explanation. Let $E \to B$ be the Banach vector bundle over $B$ with fibre over $t \in B$ given by $E_t = L^{k+1,p}(X , iT^*X_t \oplus S^+_t )$ and similarly let $F \to B$ be the Banach vector bundle over $B$ given by $F_t = L^{k,p}(X_t , i \wedge^2_+ T^*X_t \oplus S^-_t)$. Working locally on $B$ if necessary, we can assume that a smooth family of reference connections $\{ A_{0,t} \}$ have been chosen. We can further suppose that $A_{0,b} = A$. Then any $L^{k+1,p}$-configuration can be  written in the form $(A_{0,t} + a , \varphi )$, where $t \in B$ and $(a , \varphi) \in E_t$. In this way we identify $L^{k+1,p}$-configurations for the Seiberg-Witten equations with elements of $E$. The Seiberg-Witten equations for the family defines a smooth map $SW : E \to F$. Differentiating at $(A,0) \in E_b$, we get a map $dSW_{b,(A,0)} : T_{b,(A,0)} E \to T_{b,(0,0)} F_b$. Using the zero section of $F \to B$, we have a canonical splitting $T_{b,(0,0)} F_b \cong T_b B \oplus F_b$. Let $\partial_2 : T_{b,(A,0)} E \to F_b = L^{k,p}(X_b , i \wedge^2_+ T^*X_b \oplus S^-_b)$ be the composition of $dSW_{b,(A,0)}$ with the projection to $F_b$. Then the deformation complex for the families Seiberg-Witten equations at $(A,0)$ takes the form:
\[
0 \to L^{k+2,p}( X_b , i\mathbb{R}) \buildrel i \circ \partial_1 \over \longrightarrow T_{b,(A,0)} E \buildrel \partial_2 \over \longrightarrow L^{k,p}(X_b , i \wedge^2_+ T^*X_b \oplus S^-_b) \to 0,
\]
where the map $\partial_1 : L^{k+2,p}( X_b , i\mathbb{R}) \to E_b$ is the same as in the unparameterised case and $i : E_b \to T_{b,(A,0)}E$ is the natural inclusion of the vertical tangent space.\\

For $i=0,1,2$, let $\hat{H}^i( A,0)$ denote the cohomology groups of the families deformation complex. Thus $\hat{H}^1(A,0)$ is the Zariski tangent space at $(A,0)$ in the families moduli space and $\hat{H}^2(A,0)$ is the obstruction space at $(A,0)$ in the families moduli space. The inclusion of the unparameterised deformation complex into the families deformation complex gives a short exact sequence of complexes:
\begin{equation*}\xymatrix{
& 0 \ar[d] & 0 \ar[d] & 0 \ar[d] & \\
0 \ar[r] & L^{k+2,p}( X_b , i\mathbb{R}) \ar[d] \ar[r]^-{\partial_1} & E_b \ar[d]^-{i} \ar[r]^-{\partial_2} & F_b \ar[d] \ar[r] & 0 \\
0 \ar[r] & L^{k+2,p}( X_b , i\mathbb{R}) \ar[d] \ar[r]^-{i \circ \partial_1} & T_{b,(A,0)} E \ar[d] \ar[r]^-{\partial_2} & F_b \ar[d] \ar[r] & 0 \\
0 \ar[r] & 0 \ar[d] \ar[r] & T_b B \ar[d] \ar[r] & 0 \ar[d] \ar[r] & 0 \\
& 0 & 0 & 0 &
}
\end{equation*}
We obtain the following exact sequence as part of the associated long exact sequence of cohomology groups:
\[
0 \to H^1(A,0) \to \hat{H}^1(A,0) \to T_b B \buildrel \delta \over \longrightarrow H^2(A,0) \to \hat{H}^2(A,0) \to 0.
\]

\begin{lemma}
Under the isomorphism $H^2(A,0) \cong H^+(X_b) \oplus Coker(D_A)$, we have that $\delta : T_b B \to H^2(A,0)$ factors as
\[
T_b B \buildrel \delta' \over \longrightarrow H^+(X_b) \to H^+(X_b) \oplus Coker(D_A),
\]
where the second map is the natural inclusion. Furthermore, if $\mu$ is chosen to intersect the wall $\mathcal{W}$ transversally, then $\delta' : T_b B \to H^+(X_b)$ is surjective.
\end{lemma} 
\begin{proof}
Consider a tangent vector $\partial \in T_b B$. Let $\gamma(t)$ be a smooth path in $B$ with $\gamma'(0) = \partial$. We lift this to a path $\tilde{\gamma}(t)$ in $E$ passing through $( b , (A,0))$ at $t=0$. Namely, we choose $\tilde{\gamma}(t) = ( \gamma(t) , (A_{0,\gamma(t)} , 0) )$, where we recall that $A_{0,t}$ was a choice of a smooth family of reference connections. Then $\tilde{\gamma}'(0)$ is a tangent vector in $T_{b,(A,0)}E$ lifting $\partial$. The coboundary map $\delta : T_b B \to H^2(A,0)$ is given by applying $dSW_{b,(A,0)}$ to $\tilde{\gamma}'(0)$, followed by the projection $T_{b,(0,0)} F_b \to F_b$ determined by the zero section of $F \to B$ and then passing to cohomology. Applying the Seiberg-Witten equations $SW : E \to F$ to $\tilde{\gamma}(t)$ gives the path $( \gamma(t) , ( \pi_+(t) F_{A_{0,\gamma(t)}} + i\mu_{\gamma(t)} , 0 ) )$, where $\pi_+(t)$ denotes the projection $\pi_+(t) : i\wedge^2 T^*X_{\gamma(t)} \to i\wedge^2_+ T^*X_{\gamma(t)}$ to the bundle of $g_{\gamma(t)}$-self-dual $2$-forms. The fact that the spinor component of this is zero shows that $\delta$ factors through $H^+(X_b)$ as claimed. To prove the second claim, consider the map $t \mapsto \rho_t \in H^+(X_{\gamma(t)})$ given by $\rho_t = [\pi_+(t) F_{A_{0,\gamma(t)}} + i\mu_{\gamma(t)}] = -i\mathcal{W}_{\gamma(t)} + i[\mu_{\gamma(t)}]$. Then $\delta'(\partial)$ is given by differentiating $\rho_t$ at $t=0$ and projecting to $H^+(X_b)$ by using the zero section of $H^+(\mathbb{X}/B) \to B$. From this it is immediate that $\delta'$ is surjective if $[\mu]$ meets $\mathcal{W}$ transversally.
\end{proof}

\begin{corollary}\label{cor:obstructionspace}
Assume that $\dim(B) = b^+(X)$ and that $[\mu]$ meets $\mathcal{W}$ transversally. Then $\hat{H}^1(A,0) \cong Ker(D_A)$ and $\hat{H}^2(A,0) \cong Coker(D_A)$.
\end{corollary}

\subsection{A vanishing result from families Seiberg-Witten theory}\label{sec:vanishing}

Let $X$ be a smooth, compact, oriented $4$-manifold, $g$ a Riemannian metric on $X$ and $\Gamma$ a Spin$^c$-structure on $X$. Assume also that $b_1(X) = 0$. Fix a smooth reference connection $A_0$. Suppose that $\mu \in \Omega^2_+(X)$ is a perturbation lying on the wall, i.e., suppose $\mu$ is such that there exists a solution $A = A(\mu)$ to $F^+_A + i\mu = 0$. We can further assume the gauge fixing condition $d^*(A-A_0) = 0$. Since we are assuming $b_1(X) = 0$, we have that $A$ satisfying these equations is uniquely determined. Note that the complex index of the associated Dirac operator $D_A$ is given by $ind_{\mathbb{C}}(D_A) = ( c(\Gamma)^2 - \sigma(X))/8$, where $\sigma(X)$ is the signature of $X$.

\begin{lemma}\label{lem:unobstructed}
For a fixed metric $g$ and Spin$^c$-structure $\Gamma$, let $Z = Z(g,\Gamma) = \{ \mu \in \Omega^2_+(X) \; | \; \exists A \; \; F^+_A + i\mu = 0, \; \; d^*(A-A_0) = 0 \}$ be the set of perturbations lying on the wall. Let $Z_{reg} = \{ \mu \in Z \; | \; Coker(D_{A(\mu)}) = 0 \}$. Then if $c^2(\Gamma) \ge \sigma(X)$, we have that $Z_{reg}$ is a non-empty open subset of $Z$.
\end{lemma}

To prove this result we first pass to the setting of Banach manifolds. Thus we will take an integer $k \ge 0$ and a real number $p > 4$ and consider:
\[
Z^{k,p} = \{ \mu \in L^{k,p}(X , \wedge^2_+ T^*X) \; | \; \; \exists A \in L^{k+1,p}(X,\Gamma) \; \; F^+_A + i\mu = 0, \; \; d^*(A-A_0) = 0 \},
\]
where $L^{k+1,p}(X,\Gamma)$ denotes the space of $L^{k+1,p}$-Spin$^c$-connections. Note that $Z^{k,p}$ is an affine subspace of $L^{k,p}(X , \wedge^2_+ T^*X)$ of codimension $b^+(X)$, in particular it is a Banach manifold with the induced metric. The tangent space to $Z^{k,p}$ at any $\mu \in Z^{k,p}$ is given by $\{ d^+ \alpha \; | \; \alpha \in L^{k+1,p}(X , T^*X), \; \; d^* \alpha = 0 \}$. Note that since $b_1(X) = 0$, each $d^+$-exact $2$-form can be written as $d^+\alpha$ for a unique $\alpha$ satisfying $d^*\alpha = 0$.\\

We need the following two results from \cite{sal}:

\begin{lemma}[Lemma 8.17 \cite{sal}]
Assume $p > 4$ and $k \ge 0$. Suppose $A$ is an $L^{k+1,p}$ Spin$^c$-connection, $\Phi \in L^{k+1,p}(X , S^+)$ satisfying $D_A \Phi = 0$, $\Phi \neq 0$. Let $L^{k,p}_0(X , i\mathbb{R})$ denote the space of $L^{k,p}$-sections of $i\mathbb{R}$ which integrate to zero over $X$. Then the operator
\begin{equation*}
\begin{aligned}
D_{A,\Phi} : L^{k+1,p}(X , iT^*X \oplus S^+) &\to L^{k,p}_0(X , i\mathbb{R}) \oplus L^{k,p}(X , S^-) \\
\left( \begin{matrix} \alpha \\ \varphi \end{matrix} \right) & \mapsto \left( \begin{matrix} d^*\alpha \\ D_A\varphi + \alpha \cdot \Phi \end{matrix} \right)
\end{aligned}
\end{equation*}
is onto and has a right inverse.
\end{lemma}

\begin{proposition}[Proposition 8.16 \cite{sal}]
For every $p > 4$ and $k \ge 0$, let
\[
N^{k+1,p} = \{ (A , \Phi) \in L^{k+1,p}(X,\Gamma) \oplus L^{k+1,p}(X,S^+) \; | \; D_A \Phi = 0, \; \; d^*(A - A_0) = 0, \; \; \Phi \neq 0 \}.
\]
We have that $N^{k+1,p}$ is a smooth paracompact separable Banach manifold. Its tangent space at $(A,\Phi)$ is given by
\[
\{ (\alpha , \varphi ) \in L^{k+1,p}(X , iT^*X \oplus S^+) \; | \; D_A \varphi + \alpha \cdot \Phi =0, \; \; d^*\alpha = 0 \}.
\] 
\end{proposition}

Now let us consider the map $f : N^{k+1,p} \to Z^{k,p}$ given by $f(A,\Phi) = iF^+_A$. The derivative of $f$ at $(A,\Phi)$ is given by
\[
df_{(A,\Phi)}( \alpha , \varphi ) = id^+\alpha.
\]
\begin{lemma}
The map $f$ is Fredholm. More precisely, for any $(A,\Phi) \in N^{k+1,p}$ we have
\[
Ker( df_{(A,\Phi)} ) \cong Ker(D_A), \quad \quad Coker( df_{(A,\Phi)} ) \cong Coker(D_A).
\]
\end{lemma}
\begin{proof}
Let us first note that although the operator $D_A$ need not have smooth coefficients, one can use the Sobolev theorems to see that $D_A : L^{k+1,p}(X,S^+) \to L^{k,p}(X,S^-)$ is Fredholm for any $k \ge 0$ and $p > 4$ (see for example \cite[Proposition 8.2]{sal}). Thus $Ker(D_A)$ and $Coker(D_A)$ are finite dimensional. One also sees that the index of $D_A$ is independent of $A$ under the stated conditions and is therefore given by $ind_{\mathbb{C}}(D_A) = ( c(\Gamma)^2 - \sigma(X))/8$.\\

Suppose that $df_{(A,\Phi)}(\alpha , \varphi) = 0$. Then we have
\[
d^+\alpha = 0, \quad d^* \alpha = 0, \quad D_A \varphi = -\alpha \cdot \Phi.
\]
The first two of these equations say that $\alpha$ is harmonic, hence $\alpha = 0$ as $b_1(X) = 0$. The third equation then reduces to $D_A \varphi = 0$, giving $Ker( df_{(A,\Phi)} ) \cong Ker(D_A)$.\\

Let us regard $Coker(D_A)$ as the cokernel of the map $D_A : L^{k+1,p}(X,S^+) \to L^{k,p}(X,S^-)$. Let $\mu = f(A,\Phi)$ and define a map $\psi : T_\mu Z^{k,p} \to Coker(D_A)$ as follows. Recall that
\[
T_\mu Z^{k,p} \cong \{ d^+ \alpha \; | \; \alpha \in L^{k+1,p}(X , T^*X), \; \; d^* \alpha = 0 \}.
\]
Then we define $\psi( d^+\alpha)$ to be the image of $\alpha \cdot \Phi$ in $Coker(D_A)$. Let $j : Im( df_{(A,\Phi)} ) \to T_\mu Z^{k,p}$ be the inclusion map. We claim the following is a short exact sequence:
\[
0 \to Im(df_{(A,\Phi)}) \buildrel j \over \longrightarrow T_\mu Z^{k,p} \buildrel \psi \over \longrightarrow Coker(D_A) \to 0.
\]
Clearly $j$ is injective by its definition. Let $(\alpha , \varphi) \in T_{(A,\Phi)} N^{k+1,p}$. The by definition of the maps involved, we have $\psi( j( df_{(A,\Phi)} (\alpha , \varphi))) = i\psi(d^+\alpha) = i\alpha \cdot \Phi \; ({\rm mod} \; Im(D_A))$. But $(\alpha , \varphi)$ satisfies $D_A \varphi = -\alpha \cdot \Phi$, which shows that $i\alpha \cdot \Phi$ is in the image of $D_A$. Thus $\psi \circ j = 0$. Now let $d^+\alpha \in T_\mu Z^{k,p}$ and suppose that $\psi(d^+\alpha) = 0$. By definition of $\psi$, this means that $\alpha \cdot \Phi = D_A \varphi$ for some $\varphi \in L^{k+1,p}(X , S^+)$. It follows that $(-i\alpha , i\varphi) \in T_{(A,\Phi)}N^{k+1,p}$ and $j \circ df_{(A,\Phi)}(-i\alpha , i\varphi ) = d^+\alpha$. This proves exactness of the sequence at $T_\mu Z^{k,p}$. Lastly, suppose that $\Psi \in Coker( D_A )$. Represent $\Psi$ by an element of $L^{k,p}(X , S^-)$. By Lemma 8.17, there exists $(\alpha , \varphi) \in L^{k+1,p}(X , iT^*X \oplus S^+)$ with $d^* \alpha = 0$ and $D_A \varphi + \alpha \cdot \Phi = \Psi$. The latter implies that $\Psi = \alpha \cdot \Phi \; ({\rm mod} \; Im(D_A))$ and thus $\Psi = \psi( d^+ \alpha)$ proving surjectivity of the sequence at $Coker(D_A)$. It now follows that $Coker( df_{(A,\Phi)} ) \cong Coker(D_A)$.
\end{proof}

From this lemma we observe that if $(A , \Phi) \in N^{k+1,p}$, then $f$ is a submersion at $(A,\Phi)$ if and only if $Coker(D_A) = 0$. We will make use of this to prove Lemma \ref{lem:unobstructed}.\\

\begin{proof}[Proof of Lemma \ref{lem:unobstructed}] 
Let $Z^{k,p}_{reg} = \{ \mu \in Z^{k,p} \; | \; Coker(D_{A(\mu)}) = 0 \}$. We first show $Z^{k,p}_{reg}$ is non-empty. By the Sard-Smale theorem there exists a regular value $\mu \in Z^{k,p}$ of the map $f : N^{k+1,p} \to Z^{k,p}$. Observe that
\[
f^{-1}(\mu) = \{ (A(\mu) , \Phi) \in N^{k+1,p} \; | \; D_{A(\mu)} \Phi = 0, \; \; \Phi \neq 0 \}.
\]
If $f^{-1}(\mu)$ is non-empty, then there exists some $(A ,\Phi)$ for which $f$ is a submersion at $(A,\Phi)$. By the previous lemma, this implies that $Coker(D_A) = 0$.\\

Now suppose that $f^{-1}(\mu)$ is empty. This happens if and only $Ker(D_A) = 0$. It follows that $ind(D_A) \le 0$. However, we are assuming that $ind(D_A) \ge 0$, so this can happen only if $ind(D_A) = 0$. But $ind(D_A) = 0$ and $Ker(D_A) = 0$ also implies that $Coker(D_A) = 0$.\\

We now show that $Z^{k,p}_{reg}$ is open. To see this note that the assignment $\mu \mapsto D_{A(\mu)}$ is a continuous family of Fredholm operators parametrised by $Z^{k,p}$. Then $\mu \mapsto dim Coker( D_{A(\mu)})$ is upper-semicontinuous and hence $Z^{k,p}_{reg}$ is open. To complete the proof of the lemma, observe that $Z = Z^{k,p} \cap \mathcal{C}^{\infty}(X , \wedge^2_+ T^*X)$ is dense in $Z^{k,p}$, hence $Z$ meets $Z^{k,p}_{reg}$. But $Z_{reg} = Z \cap Z^{k,p}_{reg}$, which shows that $Z_{reg}$ is non-empty. It is also clear that $Z_{reg}$ is open.
\end{proof}

\begin{corollary}\label{cor:generic}
Suppose that $\pi : \mathbb{X} \to B$ is a fibrewise oriented family over a compact base of dimension $d = b^+(X)$ and fibres diffeomorphic to $X$. Let $\Gamma \in \mathcal{S}(X)$ be a monodromy invariant Spin$^c$-structure such that $c(\Gamma)^2 \ge \sigma(X)$. Then for any choice of families metric $g = \{g_b \}$ we can choose a families perturbation $\mu = \{\mu_b\}$ with the following properties:
\begin{itemize}
\item[(1)]{$\mu$ intersects $\mathcal{W}$ transversally in finitely many points $b_1, \dots , b_N \in B$.}
\item[(2)]{For $i = 1, \dots , N$, we have that $\mu_{b_i} \in Z_{reg}(g_{b_i} , \Gamma )$.}
\item[(3)]{$\mu$ is generic in the sense that every irreducible solution of the families Seiberg-Witten equations parametrised by $B$ has trivial obstruction space.}
\end{itemize}
\end{corollary}
\begin{proof}
As the set of perturbations $\mu = \{\mu_b\}$ satisfying $(3)$ is dense in the $\mathcal{C}^\infty$-topology, it suffices to show that the set of perturbations $\mu $ satisfying $(1)$ and $(2)$ is non-empty and open. Openness is straightforward so it remains to show that there exists a perturbation satisfying $(1)$ and $(2)$. Choose some initial family of perturbations $\mu' = \{ \mu'_b \}$ such that $[\mu']$ meets $\mathcal{W}$ transversally in points $b_1 , \dots , b_N \in B$. Observe for each $i$ that $\mu'_{b_i} \in Z(g_{b_i} , \Gamma)$. Notice also that $Z(g_{b_i} , \Gamma)$ is the space of $g_{b_i}$-self-dual $2$-forms on $X_{b_i}$ whose projection to $H^+_{g_{b_i}}(X_{b_i})$ agrees with $[ \mu'_{b_i}]$. It follows from Lemma \ref{lem:unobstructed} that we can find a family of $1$-forms $\{ a_b \}$ such that if we set $\mu_b = \mu'_b + d^{+_{g_b}} (a_b)$, then for each $i$, $\mu_{b_i} \in Z_{reg}( g_{b_i} , \Gamma)$. As adding a $d^{+_{g_b}}$-exact term does not change the class $[\mu']$, we see that $\mu$ satisfies $(1)$ and $(2)$.
\end{proof}

\begin{theorem}\label{t:hplusobstruction}
Let $X$ be a compact, oriented, smooth $4$-manifold with $b_1(X) = 0$ and $b^+(X) > 0$. Suppose that $\pi : \mathbb{X} \to B$ is a fibrewise oriented family over a compact base of dimension $d = b^+(X)$. Let $\Gamma \in \mathcal{S}(X)$ be a monodromy invariant Spin$^c$-structure such that $c(\Gamma)^2 > \sigma(X)$. Suppose that one of the following holds:
\begin{itemize}
\item[(1)]{we have $c^2(\Gamma) - \sigma(X) = 8 \; ({\rm mod} \; 16 )$, or}
\item[(2)]{we have that $\Gamma$ extends to a Spin$^c$-structure on $T(\mathbb{X}/B)$ (for instance, if $H^3(B , \mathbb{Z}) = 0$).}
\end{itemize}
Then $w_d( H^+(\mathbb{X}/B) ) = 0$.
\end{theorem}
\begin{proof}
Let $p : \mathcal{M} \to B$ be the families moduli space. We assume that $\mu$ is chosen to satisfy the conditions $(1),(2)$ and $(3)$ of Corollary \ref{cor:generic}. Therefore, $\mathcal{M}$ is smooth away from reducible solutions, by Corollary \ref{cor:generic} $(3)$. From by Corollary \ref{cor:generic} $(1)$, we have that $[\mu]$ meets $\mathcal{W}$ transversally in a finite number of points. Let $N$ be the number of reducible solutions. Then $w_d( H^+(\mathbb{X}/B) ) = N \; ({\rm mod} \; 2)$. Around each of the reducible solutions $(A_i , 0)$, the families obstruction space is $\hat{H}^2(A_i , 0) \cong Coker(D_{A_i})$, by Corollary \ref{cor:obstructionspace}. But $Coker(D_{A_i}) = 0$ by Corollary \ref{cor:generic} (2), hence $\hat{H}^2(A_i , 0) = 0$. It now follows from the Kuranishi model for the families moduli space around $(A_i , 0)$ that a neighbourhood of $(A_i,0)$ in $\mathcal{M}$ is homeomorphic to a cone over $\mathbb{CP}^{m-1}$, where $m = ( c^2(\Gamma) - \sigma(X) )/8$.\\

In case (1), we have that $m-1$ is even. The irreducible moduli space $\mathcal{M}^*$ gives rise to an unoriented null-cobordism of $N$ copies of $\mathbb{CP}^{m-1}$. But if $m-1$ is even, then $\mathbb{CP}^{m-1}$ is a non-trivial element in the unoriented cobordism ring, hence $N$ must be even and hence $w_d( H^+(\mathbb{X}/B) ) = 0$.\\

In case (2), we again consider the unoriented null-cobordism of $N$ copies of $\mathbb{CP}^{m-1}$ obtained from $\mathcal{M}^*$. Let $\mathcal{C}^*_B(\Gamma)$ be the families reduced configuration space modulo gauge transformations and $L_B \to \mathcal{C}^*_B(\Gamma)$ the line bundle constructed as before. The restriction of $L_B$ to $\mathcal{M}^*$ defines a line bundle $L_B \to \mathcal{M}^* $ whose restriction to each copy of $\mathbb{CP}^{m-1}$ around a reducible solution is a generator of $H^2( \mathbb{CP}^{m-1} , \mathbb{Z}_2)$, as can be seen from the Kuranishi model. The pairing of the mod $2$ cohomology class $c_1(L_B)^{m-1}$ against the $N$ copies of $\mathbb{CP}^{m-1}$ must be zero mod $2$, because of the null-cobordism. However, each copy of $\mathbb{CP}^{m-1}$ pairs with $c_1(L_B)^{m-1}$ to give $1 \; ({\rm mod} \; 2)$. Hence $N$ must be even.
\end{proof}

\begin{remark}\label{r:crestriction}
Suppose we are in the setting of Theorem \ref{t:hplusobstruction}. By monodromy invariance, $c = c(\Gamma)$ defines a section of $H^2( \mathbb{X}/B , \mathbb{R})$. If $c(\Gamma)^2 \ge 0$ and $c(\Gamma) \neq 0$, then the $L^2$-projection of $c(\Gamma)$ to $H^+(\mathbb{X}/B)$ gives a non-vanishing section, which automatically implies $w_d( H^+(\mathbb{X}/B) ) = 0$. Therefore Theorem \ref{t:hplusobstruction} is a non-trivial statement only if $c(\Gamma)^2 < 0$ or $c(\Gamma) = 0$. In such cases we also have $\sigma(X) < 0$.
\end{remark}

\section{Obstructions for actions by involutive diffeomorphisms}\label{sec:z2}

Let $f : X \to X$ be an orientation preserving diffeomorphism of $X$ with finite order $k$. Then $f$ generates an action of the finite cyclic group $\mathbb{Z}_k$ on $X$ by diffeomorphisms. We will use Theorem \ref{t:hplusobstruction} to obtain obstructions to the existence of such actions on $X$ under assumptions on how the group acts on $H^2(X , \mathbb{Z})$.\\

In this section we consider the case of involutions ($k=2$) and in the following section the case of cyclic groups of even order greater than $2$. For involutions we have:

\begin{proposition}\label{p:involution}
Let $X$ be a compact, oriented, smooth $4$-manifold with $b_1(X) = 0$ and $b^+(X) > 0$. Suppose that $f : X \to X$ is an orientation preserving involutive diffeomorphism and suppose that $\Gamma \in \mathcal{S}(X)$ is an $f$-invariant Spin$^c$-structure such that $c(\Gamma)^2 > \sigma(X)$. Then for any $f$-invariant maximal positive definite subspace $V \subseteq H^2( X , \mathbb{R})$, there is some non-zero $v \in V$ with $f(v) = v$.
\end{proposition}
\begin{proof}
To apply Theorem \ref{t:hplusobstruction}, we need a family of dimension $d = b^+(X)$. Consider the $d$-sphere $S^d$ and let $\mathbb{Z}_2$ act on $S^d$ by the antipodal map. We then take our family to be $\mathbb{X} = X \times_{\mathbb{Z}_2} S^d$ and $B = S^d/\mathbb{Z}_2 = \mathbb{RP}^d$ with $\pi : \mathbb{X} \to \mathbb{RP}^d$ the obvious projection. We have $\pi_1(B) = \mathbb{Z}_2$ and the monodromy action on $\mathcal{S}(X)$ and on the cohomology of $X$ is just the action induced by $f$. Choose a metric $g$ on $X$. By averaging we can assume $g$ is $\mathbb{Z}_2$-invariant and thus determines a fibrewise metric on the family $\mathbb{X}/B$. Furthermore, as $g$ is $\mathbb{Z}_2$-invariant, we have that $f$ acts as an involution on $H^+(X)_g$ (where the subscript means we take forms which are self-dual with respect to $g$). Then $H^+(\mathbb{X}/B)$ is just the induced flat bundle $H^+(\mathbb{X}/B) = H^+(X)_g \times_{\mathbb{Z}_2} S^d$ over $\mathbb{RP}^d$. The mod $2$ cohomology ring of $\mathbb{RP}^d$ is given by $\mathbb{R}[x]/ \langle x^{d+1} \rangle$ where $x = w_1(\mathbb{R_-}) \in H^1( \mathbb{RP}^d , \mathbb{Z}_2)$ is the first Stiefel-Whitney class of the unique non-trivial line real line bundle $\mathbb{R}_-$ over $\mathbb{RP}^d$. Suppose that $f$ acts on $H^+(X)_g$ with $u$ eigenvalues equal to $1$ and $v$ eigenvalues equal to $-1$ (so $d = u + v$). The bundle $H^+(\mathbb{X}/B)$ is isomorphic to $\mathbb{R}^u \oplus \mathbb{R}^v_-$ which has total Stiefel-Whitney class $(1+x)^v$. It follows that $H^+(\mathbb{X}/B)$ has non-trivial $d$-th Stiefel-Whitney class if and only if $u = 0$ and $v = d$, i.e. if and only if $f$ acts on $H^+(X)_g$ as $-Id$. Note that $H^3( \mathbb{RP}^d , \mathbb{Z} ) = 0$, except when $d = 3$. So in all cases where $d \neq 3$, an $f$-invariant Spin$^c$-structure will extend to a Spin$^c$-structure on $T(\mathbb{X}/B)$, so we can apply Theorem \ref{t:hplusobstruction}. In the case $d=3$ it is still true that any $f$-invariant Spin$^c$-structure on $X$ extends to a Spin$^c$-structure on $T(\mathbb{X}/B)$. To see this, one first considers the larger family $\mathbb{X}' = X \times_{\mathbb{Z}_2} S^4 \to \mathbb{RP}^4$ over $\mathbb{RP}^4$ and then restricts the family to $\mathbb{RP}^3 \subset \mathbb{RP}^4$ embedded in the standard way. The result now follows from Theorem \ref{t:hplusobstruction}.
\end{proof}

\section{Obstructions for actions by periodic diffeomorphisms of even order}\label{sec:z2m}

Let $k = 2m$ be an even integer greater than $2$. Suppose that $f : X \to X$ is an orientation preserving diffeomorphism of order $k$. Our obstruction limits the possible actions of $f$ on an $f$-invariant maximal positive definite subspace of $H^2(X , \mathbb{R})$. Recall that a real representation of $\mathbb{Z}_k$ is a direct sum of the following irreducible representations:
\begin{itemize}
\item{The trivial representation $\mathbb{R}$.}
\item{The sign representation $\mathbb{R}_-$, where the generator of $\mathbb{Z}_k$ acts by $-1$.}
\item{Let $d$ be an integer, $0 < d < k$, $d \neq m$. Then we have a $1$-dimensional complex representation $\mathbb{C}_d$ of $\mathbb{Z}_k$ where the generator acts as scalar multiplication by $e^{2\pi i \frac{d}{k}}$. The underlying real vector space of $\mathbb{C}_d$ is a rank $2$ irreducible representation of $\mathbb{Z}_k$. Notice that $\mathbb{C}_d^* \cong \mathbb{C}_{k-d}$ and that as real representations $\mathbb{C}_d \cong \mathbb{C}_{k-d}$ via complex conjugation. Thus we may assume $0 < d < m$.}
\item{Note however that as {\em oriented representations} $\mathbb{C}_d \ncong \mathbb{C}_{k-d}$, as these representations are only isomorphic by an orientation reversing map.}
\end{itemize}

\begin{proposition}\label{p:cyclic}
Let $X$ be a compact, oriented, smooth $4$-manifold with $b_1(X) = 0$ and $b^+(X) = 2u+1$ odd. Suppose that $f : X \to X$ is an orientation preserving diffeomorphism of order $k$ and suppose that $\Gamma \in \mathcal{S}(X)$ is an $f$-invariant Spin$^c$-structure such that $c(\Gamma)^2 > \sigma(X)$. Then for any $f$-invariant maximal positive definite subspace $V \subseteq H^2( X , \mathbb{R})$, the representation of $\mathbb{Z}_k$ given by $f|_V$ is not of the form
\[
\mathbb{R}_- \oplus \mathbb{C}_{d_1} \oplus \cdots \oplus \mathbb{C}_{d_u},
\]
with $d_i$ odd and $0 < d_i < m$ for all $i$. If $m$ is odd (that is, if $k = 2 \; ({\rm mod} \; 4)$), then more generally the representation of $\mathbb{Z}_k$ given by $f|_V$ is not of the form
\[
\mathbb{R}_-^{2a+1} \oplus \mathbb{C}_{d_1} \oplus \cdots \oplus \mathbb{C}_{d_b},
\]
with $d_i$ odd, $0 < d_i < m$ for all $i$ and where $a+b = u$.
\end{proposition}
\begin{proof}
As in the proof of Proposition \ref{p:involution}, we need a family of dimension $d = b^+(X) = 2u+1$. Consider a lens space $L^{2u+1} = L(k ; 1 , \dots , 1) = S^{2u+1}/\mathbb{Z}_k$. We take our family to be $\mathbb{X} = X \times_{\mathbb{Z}_k} S^d$ and $B = L^{2u+1}$, so that $\pi_1(B) = \mathbb{Z}_k$. Choose a metric $g$ on $X$. By averaging we can assume $g$ is $\mathbb{Z}_k$-invariant. Then $f$ acts on $H^+(X)_g$ and $H^+(\mathbb{X}/B)$ is the induced flat bundle $H^+(\mathbb{X}/B) = H^+(X)_g \times_{\mathbb{Z}_k} S^d$. As $k$ is even, one finds that the mod $2$ cohomology of $L^{2u+1}$ is $\mathbb{Z}_2$ in degrees $0 , 1 , \dots , 2u+1$. Let $\alpha \in H^1(L^{2u+1} , \mathbb{Z}_2)$ and $\beta \in H^2(L^{2u+1} , \mathbb{Z}_2)$ be generators in degrees $1$ and $2$. Then for $0 \le j \le 2u+1$, we have that $H^j( L^{2u+1} , \mathbb{Z}_2)$ is generated by $\beta^i$ if $j = 2i$ and by $\alpha \cup \beta^i$ if $j = 2i+1$. To each representation $V$ of $\mathbb{Z}_k$, we obtain an associated flat bundle $E(V) = V \times_{\mathbb{Z}_k} S^d$ over $L^{2u+1}$. The sign representation $\mathbb{R}_-$ defines the unique non-trivial real line bundle on $L^{2u+1}$. Therefore the total Stiefel-Whitney class of $E(\mathbb{R}_-)$ is $w( E(\mathbb{R}_-) ) = 1 + \alpha$. Note that the lens space $L^{2u+1}$ is a circle bundle $q : L^{2u+1} \to \mathbb{CP}^u$ over $\mathbb{CP}^u$. The representation $\mathbb{C}_d$ defines a complex line bundle on $L^{2u+1}$ which is easily seen to be $E(\mathbb{C}_d) = q^* \mathcal{O}(d)$. Moreover, one sees via the Gysin sequence for $q : L^{2u+1} \to \mathbb{CP}^u$ that the first Chern class of $q^* \mathcal{O}(1)$ taken mod $2$ equals $\beta$. It follows that the total Stiefel-Whitney class of $E(\mathbb{C}_d)$ is $w( E(\mathbb{C}_d) ) = (1 + d \beta)$. Now if the representation of $\mathbb{Z}_k$ on $H^+(X)_g$ given by $f$ is of the form
\[
\mathbb{R}_- \oplus \mathbb{C}_{d_1} \oplus \cdots \oplus \mathbb{C}_{d_u},
\]
with $d_i$ odd and $0 < d_i < m$ for all $i$. Then it follows that the total Stiefel-Whitney class of the associated bundle $H^+(\mathbb{X}/B)$ is
\[
w( H^+( \mathbb{X}/B ) ) = (1+\alpha)(1+d_1 \beta) \cdots (1 + d_u \beta)
\]
and in particular
\[
w_d( H^+( \mathbb{X}/B ) ) = d_1 d_2 \cdots d_u \alpha \cup \beta^u.
\]
Thus, if $d_1 , \dots , d_u$ are all odd then $w_d( H^+(\mathbb{X}/B) ) \neq 0$. In the case that $m$ is odd then it can be shown that $\alpha^2 = \beta$. Therefore if the action of $\mathbb{Z}_k$ on $H^+(X,g)$ given by $f$ is of the form
\[
\mathbb{R}_-^{2a+1} \oplus \mathbb{C}_{d_1} \oplus \cdots \oplus \mathbb{C}_{d_b},
\]
with $d_i$ odd, $0 < d_i < m$ for all $i$ and where $a+b = u$, then by a computation similar to the above, we find that $w_d( H^+(\mathbb{X}/B) ) \neq 0$.\\

Next, we observe that $H^3( L^{2u+1} , \mathbb{Z} ) = 0$, except when $2u+1 = 3$. So in all cases where $2u+1 \neq 3$, an $f$-invariant Spin$^c$-structure will extend to a Spin$^c$-structure on $T(\mathbb{X}/B)$. However, as in the proof of Proposition \ref{p:involution}, when $2u+1 = 3$, we can embed our family into a larger family given by a higher-dimensional lens space, where it is clear that the Spin$^c$-structure extends. We see that Theorem \ref{t:hplusobstruction} can be applied in all cases and so the result follows.
\end{proof}

\section{Obstructions for actions by commuting diffeomorphisms}\label{sec:freeab}

Let $f_1 , \dots , f_d : X \to X$ be $d$ commuting orientation preserving diffeomorphisms of $X$ and assume that $d = b^+(X)$. Then $f_1 , \dots , f_d$ generate an action of the group $\mathbb{Z}^d$ on $X$ by diffeomorphisms. Since the group $\mathbb{Z}^d$ is non-compact, we can not assume that $X$ has a $\mathbb{Z}^d$-invariant metric. As a replacement, we will need to assume that $f_1 , \dots , f_d$ preserve a maximal positive definite subspace $V \subseteq H^2( X , \mathbb{R})$. Thus $V$ has the structure of a real orthogonal representation of $\mathbb{Z}^d$. In order to get a non-trivial obstruction, we are lead to consider only those representations where each $f_i$ acts on $V$ with eigenvalues $\pm 1$. For such a representation $V$ can be simultaneously diagonalised. That is, there exists a basis in which
\begin{equation}\label{equ:diag}
f_i = diag( (-1)^{\epsilon^1_i} , \dots , (-1)^{\epsilon^d_i} ),
\end{equation}
for some $\epsilon^j_i \in \mathbb{Z}_2 = \{ 0 , 1\}$.

\begin{proposition}\label{p:commutingdiffeos}
Let $X$ be a compact, oriented, smooth $4$-manifold with $b_1(X) = 0$ and $b^+(X) = d > 0$. Suppose that $f_1 , \dots , f_d : X \to X$ are commuting diffeomorphisms and suppose that $\Gamma \in \mathcal{S}(X)$ is a Spin$^c$-structure on $X$ which is preserved by $f_1 , \dots , f_d$ and satisfies $c(\Gamma)^2 > \sigma(X)$. If $d \ge 3$, assume in addition that $c^2(\Gamma) - \sigma(X) = 8 \; ({\rm mod} \; 16)$. Suppose that $V \subseteq H^2(X , \mathbb{R})$ is a maximal positive definite subspace preserved by $f_1 , \dots , f_d$. Assume also that in some basis of $V$, the $f_i$ are given by Equation (\ref{equ:diag}). Then $det( \epsilon^j_i) = 0$, where we view $\epsilon^j_i$ as elements of the field $\mathbb{Z}_2 = \{0,1\}$.
\end{proposition}
\begin{proof}
Consider the family $\pi : \mathbb{X} \to B = T^d$ given by $\mathbb{X} = X \times_{\mathbb{Z}^d} \mathbb{R}^d$. Then $V$ determines a maximal positive definite sub-bundle $\widetilde{V}$ of $H^2( \mathbb{X}/B , \mathbb{R})$. If $d=1$ or $2$ then we can apply Theorem \ref{t:hplusobstruction} $(2)$, since $H^3(B,\mathbb{Z}) = 0$. If $d \ge 3$, then our assumptions are such that we may apply Theorem \ref{t:hplusobstruction} $(1)$. In either case we deduce that $w_d( \widetilde{V} ) = 0 \in H^d( T^d , \mathbb{Z}_2)$. Next view $T^d$ as the product of $d$ circles and $\pi_i : T^d \to S^1$ the projection to the $i$-th circle. Let $x \in H^1(S^1 , \mathbb{Z}_2) \cong \mathbb{Z}_2$ be the generator and set $x_i = \pi_i^*(x)$. Then $H^*( T^d , \mathbb{Z}_2)$ is an exterior algebra over $\mathbb{Z}_2$ in $x_1 , \dots , x_d$. Now one easily sees that if the $f_i$ act on $V$ according to Equation \ref{equ:diag}, then
\[
w_d( \widetilde{V} ) = ( \epsilon_1^1 x_1 + \cdots + \epsilon^1_d x_d) \cdots ( \epsilon_1^d x_1 + \cdots + \epsilon^d_d x_d) = det(\epsilon^j_i) x_1 \cdots x_d,
\]
so that $w_d( \widetilde{V} ) = 0$ if and only if $det(\epsilon^j_i) = 0$.
\end{proof}

\begin{remark}
Proposition \ref{p:commutingdiffeos} generalises Theorems 1.1 and 1.2 of \cite{na1} which concern non-smoothability of $\mathbb{Z}$ and $\mathbb{Z}^2$-actions on $4$-manifolds. However the assumptions placed on $X$ were more restrictive than for those of our proposition.
\end{remark}

\section{Obstructions to smooth $\mathbb{Z}_2 \times \mathbb{Z}_2$-actions}\label{sec:z2z2}

One can use families Seiberg-Witten theory to obtain an obstruction to smooth actions by a product of finite cyclic groups of even order. For simplicity, we will restrict ourselves to the simple case of $\mathbb{Z}_2 \times \mathbb{Z}_2$-actions. 

\begin{proposition}\label{p:z2z2}
Let $X$ be a compact, oriented, smooth $4$-manifold with $b_1(X) = 0$ and $b^+(X) > 0$. Suppose that $f_1 , f_2 : X \to X$ are orientation preserving involutive diffeomorphisms which commute and thus define a $\mathbb{Z}_2 \times \mathbb{Z}_2$-action. Suppose that $\Gamma \in \mathcal{S}(X)$ is a Spin$^c$-structure which is invariant under $f_1$ and $f_2$ which satisfies $c(\Gamma)^2 > \sigma(X)$ and $c(\Gamma)^2 - \sigma = 8 \; ({\rm mod} \; 16)$. Then for any maximal positive definite subspace $V \subseteq H^2( X , \mathbb{R})$ preserved by $f_1,f_2$, there is some non-zero $v \in V$ with $f_1(v) = f_2(v) = v$.
\end{proposition}
\begin{proof}
If $b^+(X) = 1$, then the result follows from Proposition \ref{p:involution} applied separately to $f_1$ and $f_2$. Thus assume $b^+(X) \ge 2$. Choose positive integers $d_1,d_2$ such that $d_2 + d_2 = d = b^+(X)$. Let $\mathbb{Z}_2 \times \mathbb{Z}_2$ act on $S^{d_1} \times S^{d_2}$, where the first $\mathbb{Z}_2$ acts on $S^{d_1}$ by the antipodal map and similarly the second $\mathbb{Z}_2$ acts by the antipodal map on $S^{d_2}$. Consider the family $\mathbb{X} = X \times_{\mathbb{Z}_2 \times \mathbb{Z}_2} S^{d_1} \times S^{d_2} \to B = \mathbb{RP}^{d_1} \times \mathbb{RP}^{d_2}$. Let $\widetilde{V} = V \times_{\mathbb{Z}_2 \times \mathbb{Z}_2} S^{d_1} \times S^{d_2} \to B$ be the vector bundle associated to $V$. Our assumptions are such that we may apply Theorem \ref{t:hplusobstruction} to deduce that $w_d(\widetilde{V}) = 0$. For $\epsilon_1 , \epsilon_2 \in \{ 1 , -1 \}$, let $\mathbb{R}_{\epsilon_1 , \epsilon_2}$ be the $1$-dimensional real representation of $\mathbb{Z}_2 \times \mathbb{Z}_2 = \langle f_1 \rangle \times \langle f_2 \rangle$, where $f_i$ acts as multiplication by $\epsilon_i$. Let $\widetilde{\mathbb{R}}_{\epsilon_1 , \epsilon_2} \to B$ be the real line bundle on $B$ associated to $\mathbb{R}_{\epsilon_1 , \epsilon_2}$. Note that $V$ can be written as a direct sum of such $1$-dimensional representations and thus $\widetilde{V}$ is a direct sum of the associated line bundles. Next, observe that $H^*(B , \mathbb{Z}_2) \cong \mathbb{Z}_2[x,y]/\langle x^{d_1+1}, y^{d_2+1} \rangle$, where $x,y \in H^1( B , \mathbb{Z}_2)$ correspond to the generators of $H^1( \mathbb{RP}^{d_1} , \mathbb{Z}_2)$ and $H^1( \mathbb{RP}^{d_2} , \mathbb{Z}_2)$ respectively. Clearly, we have:
\[
w( \widetilde{\mathbb{R}}_{-1 , 1} ) = 1+x, \quad w( \widetilde{\mathbb{R}}_{1 , -1} ) = 1+y, \quad w( \widetilde{\mathbb{R}}_{-1 , -1} ) = 1+x+y.
\]
Now suppose that $V$ decomposes as:
\[
V = \mathbb{R}_{1,1}^{p} \oplus \mathbb{R}_{-1,1}^q \oplus \mathbb{R}_{1,-1}^r \oplus \mathbb{R}_{-1,-1}^s
\]
where $p+q+r+s = d$. Then $\widetilde{V}$ similarly decomposes and thus
\[
w( \widetilde{V}) = (1+x)^q(1+y)^r(1+x+y)^s.
\]
In particular, we find that
\begin{equation}\label{equ:qrs}
w_d( \widetilde{V} ) = x^q y^r (x+y)^s.
\end{equation}
The statement of the proposition is equivalent to saying $p \neq 0$. Suppose on the contrary that $p = 0$, so that $q+r+s = d$. We can also assume that $q \neq 0$, for if $p = q = 0$, then $f_1$ acts as the identity on $V$ and the proposition follows by applying Proposition \ref{p:involution} to $f_2$. Similarly, we can assume $r \neq 0$. Let us choose $d_1 = q$ and $d_2 = r+s$. Then $x^{q+1} = 0$, so we deduce from (\ref{equ:qrs}) that $w_d(\widetilde{V}) = x^q y^{r+s} = x^{d_1} y^{d_2} \neq 0$, a contradiction. Hence $p \neq 0$, as required.
\end{proof}

\section{Some applications}\label{sec:appl}

We consider some applications of the obstruction theorems obtained in the previous sections. Many further variations of the examples presented here can be constructed.

\subsection{$\mathbb{Z}_2$-actions}

Let $f : H^2(X , \mathbb{Z}) \to H^2(X , \mathbb{Z})$ be an involutive isometry. Proposition \ref{p:involution} can be interpreted as an obstruction to realising $f$ as the map induced by an involutive diffeomorphism of $X$ (note that any such diffeomorphism must be orientation preserving, provided $b_2(X) \neq 0$). For simplicity we consider the case where $X$ is simply connected. Then an $f$-invariant Spin$^c$-structure is uniquely determined by an $f$-invariant integral lift $c \in H^2(X , \mathbb{Z})$ of $w_2(TX)$.\\

Case (i): $X$ is spin. Then $w_2(TX) = 0$ and $c = 0$ is an invariant integral lift. Suppose in addition that $\sigma(X) < 0$. Proposition \ref{p:involution} implies the following:
\begin{proposition}\label{prop:involminus1}
Let $X$ be a compact, smooth, simply-connected spin $4$-manifold with $\sigma(X) < 0$. Let $f : H^2(X , \mathbb{Z}) \to H^2(X , \mathbb{Z})$ be an involutive isometry. If $f$ acts on a maximal positive definite subspace of $H^2(X , \mathbb{R})$ as $-Id$, then $f$ can not be realised by an involutive diffeomorphism of $X$. 
\end{proposition}

As a special case we obtain:
\begin{corollary}
Let $X$ be a compact, smooth, simply-connected spin $4$-manifold with $\sigma(X) \neq 0$. Then there does not exist an involutive diffeomorphism of $X$ which acts as minus the identity on $H^2(X , \mathbb{Z})$.
\end{corollary}
The condition that $\sigma(X) \neq 0$ in the above corollary is necessary, for example the involution on $S^2 \times S^2$ which acts as the antipodal map on each $S^2$-factor acts as $-Id$ on $H^2( S^2 \times S^2 , \mathbb{Z})$.\\

Case (ii): $X$ is not spin. Then the intersection form of $X$ is diagonal. That is, there is an integral basis in which the intersection form is the diagonal matrix $diag( 1  , \dots , 1 , -1 , \dots , -1)$ with $b^+(X)$ entries equal to $1$ and $b^-(X)$ entries equal to $-1$. According to Remark \ref{r:crestriction}, in order to obtain non-trivial results we may as well assume $\sigma(X) < c^2 \le 0$. To obtain a non-existence result via Proposition \ref{p:involution}, we need to find involutive isometries $f : H^2(X , \mathbb{Z}) \to H^2(X , \mathbb{Z})$ such that $f(c) = c$ for some $c \in H^2(X , \mathbb{Z})$ with $c = (1  , 1 , \dots , 1) \; ({\rm mod} \; 2)$, $\sigma(X) < c^2 \le 0$ and such that $f$ acts as $-Id$ on some maximal positive definite subspace of $H^2(X , \mathbb{R})$. Such an isometry can not be realised as an involutive diffeomorphism of $X$. Suppose $e \in H^2(X , \mathbb{Z})$ satisfies $e^2 = 2$ and let $r_e : H^2(X , \mathbb{Z} ) \to H^2(X , \mathbb{Z})$ be the corresponding reflection
\[
r_e(x) = x - \langle x , e \rangle e,
\]
then $r_e$ is an involutive isometry and $r_e(e) = -e$. Using this, we obtain:
\begin{proposition}\label{prop:reflections}
Let $X$ be a compact, smooth, simply-connected, non-spin $4$-manifold with $b^+(X) > 0$ (so $X$ is homeomorphic to $\# m \mathbb{CP}^2 \# n \overline{\mathbb{CP}^2}$ for some $m,n$ with $m > 0$). Suppose that there exists $e_1 , \dots , e_m \in H^2(X , \mathbb{Z})$ with $e_i^2 = 2$ for all $i$ and $\langle e_i , e_j \rangle = 0$ for all $i \neq j$. Suppose there is also a characteristic $c \in H^2(X , \mathbb{Z})$ with $\sigma(X) < c^2 \le 0$ such that $\langle c , e_i \rangle = 0$. Then there does not exist an involutive diffeomorphism of $X$ which acts on $H^2(X , \mathbb{Z})$ as the product of reflections $r_{e_1} r_{e_2} \cdots r_{e_m}$.
\end{proposition}
\begin{proof}
Let $f = r_{e_1} r_{e_2} \cdots f_{e_m}$. Then $f$ is an involutive isometry and acts as $-Id$ on the maximal positive definite subspace spanned by $e_1 , \dots ,e_m$. Moreover $f(c) = c$ and $c^2 > \sigma$. Thus by Proposition \ref{p:involution} we see that $f$ can not be realised by an involutive diffeomorphism of $X$.
\end{proof}

\begin{example}\label{ex:e1e2}
Many examples of $c , e_1 , \dots , e_m$ satisfying the assumptions of Proposition \ref{prop:reflections} can be constructed. For instance, if $X = \# 2 \mathbb{CP}^2 \# 11 \overline{\mathbb{CP}^2}$ we have $H^2(X , \mathbb{Z}) = 2(1) \oplus 11(-1)$ and we can take
\begin{equation*}
\begin{aligned}
c &= (3, 1 ; 1 , \dots , 1)\\
e_1 &= (0 , 2 ; 1 , 1 , 0 , \dots , 0) \\
e_2 &= (6,1 ; 1 , 1 , 2 , 2 , 2 , 2 , 2 , 2 , 2 , 2 , 1).
\end{aligned}
\end{equation*}
We conclude that for any smooth structure on $X$, there is no differentiable involution inducing $r_{e_1} r_{e_2}$ on $H^2(X , \mathbb{Z})$.
\end{example}

\begin{example}
Another method to construct involutions which can not be realised by diffeomorphisms is as follows: Let $P , Q$ be positive definite unimodular symmetric bilinear forms of positive rank. Suppose that $P$ is even and $Q$ is odd. Then $P \oplus (-Q)$ is an odd indefinite unimodular symmetric bilinear form, so must be isomorphic to $m (1) \oplus n (-1)$, where $m = rk(P)$, $n = rk(Q)$. In particular, $P \oplus (-Q)$ is the intersection form of the topological $4$-manifold $X = \# m \mathbb{CP}^2 \# n \overline{\mathbb{CP}^2}$. Fix an identification $H^2(X , \mathbb{Z}) = P \oplus (-Q)$ and let $f : H^2(X , \mathbb{Z}) \to H^2(X , \mathbb{Z})$ be the involution given by $f(p,q) = (-p , q)$. Then $f$ acts as $-Id$ on a maximal positive definite subspace of $H^2(X , \mathbb{R})$, namely $P \otimes \mathbb{R}$. Suppose that $Q$ has a characteristic $c \in Q$ such that $c^2 < rk(Q) - rk(P)$. Then $\tilde{c} = (0,c)$ is an $f$-invariant characteristic of $H^2(X , \mathbb{Z})$ which satisfies $\tilde{c}^2 > \sigma(X)$. Then by Proposition \ref{p:involution}, $f$ can not be realised by an differentiable involution on $X$, for any smooth structure.
\end{example}

Denote by $E_8$ the unique compact simply connected topological $4$-manifold with intersection form $E_8$ and let $-E_8$ denote the same manifold but with the opposite orientation.

\begin{proposition}\label{p:z2nonrealisation}
Let $X$ be the topological $4$-manifold $X = \# a (S^2 \times S^2) \# 2b(-E_8)$ where $a > 3b$ and $b\ge 1$. Then $H^2(X , \mathbb{Z})$ admits an involution $f : H^2(X , \mathbb{Z}) \to H^2(X , \mathbb{Z})$ with the following properties:
\begin{itemize}
\item[(i)]{$f$ can be realised by the induced action of a continuous, locally linear involution $X \to X$.}
\item[(ii)]{$f$ can be realised by the induced action of a diffeomorphism $X \to X$, where the smooth structure is obtained by viewing $X$ as $\# (a-3b)(S^2 \times S^2) \# b(K3)$.}
\item[(iii)]{$f$ can not be realised by the induced action of an involutive diffeomorphism $X \to X$ for any smooth structure on $X$.}
\end{itemize}
\end{proposition}

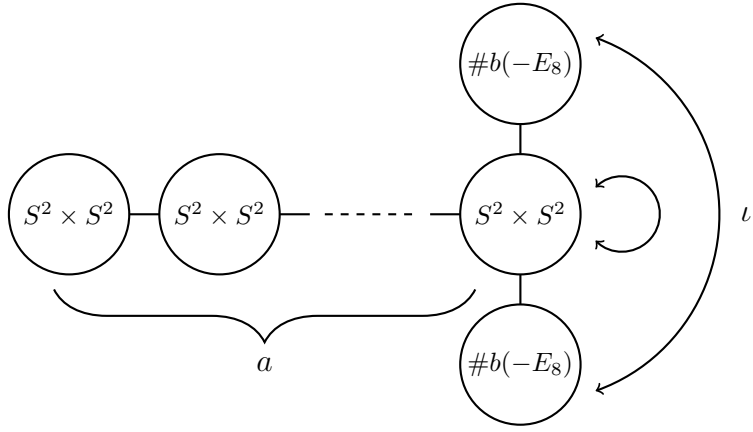
\begin{figure}
\begin{tikzpicture}

\draw[thick] (0,0) circle (0.8);
\node at (0,0) {$S^2 \times S^2$};

\draw[thick] (2,0) circle (0.8);
\node at (2,0) {$S^2 \times S^2$};

\draw[thick] (6,0) circle (0.8);
\node at (6,0) {$S^2 \times S^2$};

\draw[thick] (6,2) circle (0.8);
\node at (6,2) {$\# b(-E_8)$};

\draw[thick] (6,-2) circle (0.8);
\node at (6,-2) {$\#b(-E_8)$};

\draw[thick](0.8,0) -- (1.2,0); 
\draw[thick](2.8,0) -- (3.2,0); 
\draw[thick, dashed](3.4,0) -- (4.6,0); 
\draw[thick](4.8,0) -- (5.2,0); 

\draw[thick](6,0.8) -- (6,1.2); 
\draw[thick](6,-0.8) -- (6,-1.2);

\draw[thick, <->]
(7 , -0.35) arc (-135:135: 0.5 and 0.5);

\draw[thick, <->]
(7 , -2.35) arc (-70:70: 2.5 and 2.5);

\node at (9,0) {{\Large $\iota$}};

\draw [snake=brace,thick,segment amplitude=20pt,segment length=20pt] (5.4 , -1) -- (-0.2, -1);

\node at (2.6 , -2){{\Large $a$}};

\end{tikzpicture}
\caption{The $4$-manifold $X = \# a (S^2 \times S^2) \# 2b(-E_8)$ with $\mathbb{Z}_2$-action.}\label{fig:1}
\end{figure}

\begin{proof}
We explain the construction of the desired $4$-manifold $X$, which is illustrated in Figure \ref{fig:1}. First consider $S^2 \times S^2 = \mathbb{CP}^1 \times \mathbb{CP}^1$ equipped with the involution $\iota : S^2 \times S^2 \to S^2 \times S^2$ which acts as complex conjugation on each $\mathbb{CP}^1$ factor. This action has multiple fixed points and by considering the action of $\iota$ on the tangent space of any fixed point, we see that it is possible to take an equivariant connected sum of $(S^2 \times S^2 , \iota)$ with itself any number of times to obtain an involution $\iota$ on $\# a (S^2 \times S^2)$ which acts as $-Id$ on $H^2( \# a(S^2 \times S^2) , \mathbb{Z} )$. Choose a point $x \in \# a (S^2 \times S^2)$ which is not fixed by $\iota$ and attach copies of $-E_8$ at $x$ and $\iota(x)$. Then we can extend $\iota$ to an involution on $\# a(S^2 \times S^2) \# 2(-E_8)$ which swaps the two copies of $-E_8$. Repeating this $b$ times we obtain a continuous involution $\iota$ on $X = \# a (S^2 \times S^2) \# 2b(-E_8)$. Let $f = \iota_* : H^2(X , \mathbb{Z}) \to H^2(X , \mathbb{Z})$ be the induced involution on $H^2(X , \mathbb{Z})$. Then (i) holds by our construction of $f$. The connected sum decomposition of $X$ gives an identification $H^2(X , \mathbb{Z}) = a H \oplus 2b (-E_8)$, where we define $H$ to be the intersection form of $S^2 \times S^2$. By construction of $\iota$, $f$ acts as $-Id$ on $a H$ and swaps pairs of copies of $-E_8$. If $a > 3b$ then $X$ admits at least one smooth structure since we can view $X$ as $\# (a-3b)(S^2 \times S^2) \# b(K3)$. Moreover if $a > 3b$ and $b \ge 1$, then with respect to this smooth structure $X$ is of the form $X = (S^2 \times S^2) \# N$, where $N$ is a smooth simply-connected compact $4$-manifold with indefinite intersection form. By a theorem of Wall \cite[Theorem 2]{wall}, every isometry of $H^2(X , \mathbb{Z})$ is realised by a diffeomorphism. In particular $f$ is realised by some diffeomorphism of $X$, which proves (ii). Part (iii) follows from our obstruction theorem for involutive diffeomorphisms, or more specifically, Proposition \ref{prop:involminus1}.
\end{proof}

\subsection{$\mathbb{Z}_{2k}$-actions}

Let $k \ge 2$ be an integer. We consider $\mathbb{Z}_{2k}$-actions on $4$-manifolds.

\begin{proposition}\label{p:z2knonrealisation}
Let $X$ be the topological $4$-manifold $X = \# a (S^2 \times S^2) \# 2kb(-E_8)$ where $a > 3kb$, $b\ge 1$ and suppose that $a$ and $k$ are odd. Then $H^2(X , \mathbb{Z}) = a H \oplus 2kb(-E_8)$. Let $f : H^2(X , \mathbb{Z}) \to H^2(X , \mathbb{Z})$ be an isometry of order $2k$ which acts as $\left( \begin{matrix} 0 & -1 \\ -1 & 0 \end{matrix} \right)$ on each $H$ summand and acts as a permutation of the $-E_8$ summands such that each cycle of the permutation has length $2k$. Then $f$ has the following properties:
\begin{itemize}
\item[(i)]{$f$ can be realised by a continuous, locally linear action of $\mathbb{Z}_{2k}$ on $X$.}
\item[(ii)]{$f$ can be realised by the induced action of a diffeomorphism $X \to X$, where the smooth structure is obtained by viewing $X$ as $\# (a-3kb)(S^2 \times S^2) \# bk(K3)$.}
\item[(iii)]{$f$ can not be realised by a smooth $\mathbb{Z}_{2k}$-action for any smooth structure on $X$.}
\end{itemize}
\end{proposition}

\begin{proof}
We will construct the desired $4$-manifold $X$, which is illustrated in Figure \ref{fig:2}. First of all we consider two different $\mathbb{Z}_{2k}$-actions on $S^2 \times S^2$. Then we take an equivariant connected sum $\# a(S^2 \times S^2)$, where the $\mathbb{Z}_{2k}$-action on each $S^2 \times S^2$ factor alternates between the two types. Finally we attach to this $2kb$ copies of $-E_8$.\\

We now describe the two different $\mathbb{Z}_{2k}$-actions on $S^2 \times S^2$. We will denote the generators of these two actions by $\iota_1$ and $\iota_2$.\\

{\em Construction of $\iota_1$}: let $r,s : \mathbb{R}^2 \to \mathbb{R}^2$ be a pair of reflections generating the dihedral group of order $2k$. So $r^2 = s^2 = 1$ and $(sr)^k = 1$. Let $r,s$ act on the unit sphere $S^2 \subset \mathbb{R}^3$ by $r(x,y,z) = ( r(x,y) , z )$, $s(x,y,z) = (s(x,y) , z)$. The two fixed points of this action on $S^2$ are $u = (0,0,1)$ and $v = (0,0,-1)$. Now define the diffeomorphism $\iota_1 : S^2 \times S^2 \to S^2 \times S^2$ to be given by $\iota_1(x,y) = (r(y) , s(x))$. Clearly $\iota_1$ generates a $\mathbb{Z}_{2k}$-action on $S^2 \times S^2$. We see that $\iota_1$ has precisely two fixed points, which are $(u,u)$ and $(v,v)$. We obtain induced actions of $\mathbb{Z}_{2k}$ on the tangent spaces $T_{(u,u)} (S^2 \times S^2)$ and $T_{(v,v)} (S^2 \times S^2)$ of the fixed points. By a direct calculation one finds, as oriented representations of $\mathbb{Z}_{2k}$, that:
\begin{equation}\label{equ:rep1}
T_{(u,u)} (S^2 \times S^2) \cong T_{(v,v)} (S^2 \times S^2) \cong \mathbb{C}_1 \oplus \mathbb{C}_{-1+k}.
\end{equation}
We note that the action of $\iota_1$ on $H = H^2(S^2 \times S^2 , \mathbb{Z})$ is of the form $\left[ \begin{matrix} 0 & -1 \\ -1 & 0 \end{matrix} \right]$.\\

{\em Construction of $\iota_2$}: let $\varphi : \mathbb{R}^2 \to \mathbb{R}^2$ be a rotation of order $2k$. We let $\varphi$ act on the unit sphere $S^2 \subset \mathbb{R}^3$ by $\varphi(x,y,z) = ( \varphi(x,y) , z)$. Let $\alpha : S^2 \times S^2$ be the antipodal map and note that $\alpha$ commutes with $\varphi$. We let $\iota_2 : S^2 \times S^2$ be given by $\iota_2(x,y) = (\varphi \alpha y , \varphi \alpha x)$. Then $\iota_2$ has two fixed points, which are $(u,v)$ and $(v,u)$. A short calculation shows that, as oriented representations of $\mathbb{Z}_{2k}$, we have:
\begin{equation}\label{equ:rep2}
T_{(u,v)} (S^2 \times S^2) \cong T_{(v,u)} (S^2 \times S^2) \cong \mathbb{C}_1 \oplus \mathbb{C}_{1+k}.
\end{equation}

\begin{figure}
\begin{tikzpicture}

\draw[thick] (0,0) circle (0.8);
\node at (0,0) {$S^2 \times S^2$};

\draw[thick] (2,0) circle (0.8);
\node at (2,0) {$S^2 \times S^2$};

\draw[thick] (6,0) circle (0.8);
\node at (6,0) {$S^2 \times S^2$};

\draw[thick] (7.2,1.6) circle (0.8);
\node at (7.2,1.6) {$\# b(-E_8)$};

\draw[thick] (7.2,-1.6) circle (0.8);
\node at (7.2,-1.6) {$\#b(-E_8)$};

\draw[thick](0.8,0) -- (1.2,0); 
\draw[thick](2.8,0) -- (3.2,0); 
\draw[thick, dashed](3.4,0) -- (4.6,0); 
\draw[thick](4.8,0) -- (5.2,0); 

\draw[thick](6.48,0.64) -- (7.2-0.48,1.6-0.64);
\draw[thick](6.48,-0.64) -- (7.2-0.48,-1.6+0.64);

\draw [fill] (7.4,0.4) circle(0.04);
\draw [fill] (7.4,0.0) circle(0.04);
\draw [fill] (7.4,-0.4) circle(0.04);

\draw[thick, ->]
(9.5, -0.45) arc (-170:170: 0.5 and 2);

\node at (11,0) {{\Large $\iota$}};

\draw [snake=brace,thick,segment amplitude=10pt,segment length=10pt] (8.3 , 2) -- (8.3, -2);

\node at (9,0){{\Large $2k$}};

\draw [snake=brace,thick,segment amplitude=20pt,segment length=20pt] (5.4 , -1) -- (-0.2, -1);

\node at (2.6 , -2){{\Large $a$}};

\node at (0,1.5){{\Large $(u,u)$}};
\node at (1.4,1.5){{\Large $(u,v)$}};
\node at (2.8,1.5){{\Large $(v,u)$}};

\draw[thin, ->](0.2 , 1.2) -- (0.7,0.1);
\draw[thin, ->](1.4 , 1.2) -- (1.3,0.1);
\draw[thin, ->](2.8 , 1.2) -- (2.7,0.1);

\node at (4,1.5){{\Large $\cdots$}};

\end{tikzpicture}
\caption{The $4$-manifold $X = \# a (S^2 \times S^2) \# 2kb(-E_8)$ with $\mathbb{Z}_{2k}$-action.}\label{fig:2}
\end{figure}
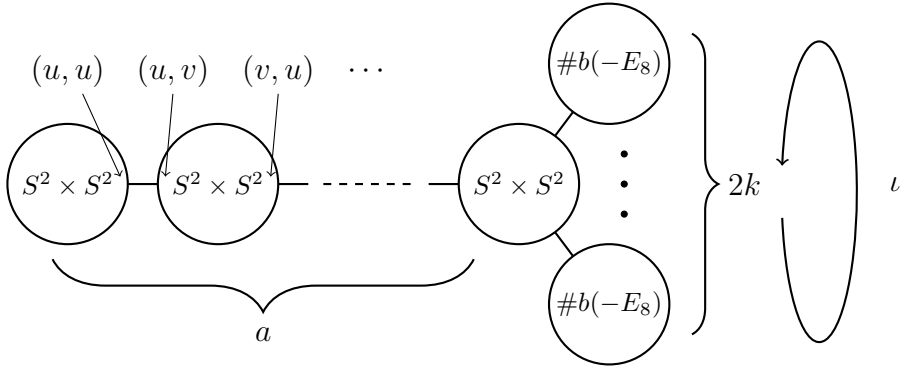

From (\ref{equ:rep1}) and (\ref{equ:rep2}), we see that the $\mathbb{Z}_{2k}$-actions of $\iota_1 , \iota_2$ on the tangent spaces of their stabilisers are isomorphic by an orientation reversing isomorphism. We note that the action of $\iota_2$ on $H = H^2(S^2 \times S^2 , \mathbb{Z})$ is of the form $\left[ \begin{matrix} 0 & -1 \\ -1 & 0 \end{matrix} \right]$.\\

We form an equivariant connected sum $( \# a (S^2 \times S^2) , \iota )$ as follows. Let $(S^2 \times S^2)_1 , \dots , (S^2 \times S^2)_a$ be the $a$ copies of $S^2 \times S^2$. We let $\iota_1$ act on $(S^2 \times S^2)_i$ for odd $i$ and let $\iota_2$ act on $(S^2 \times S^2)_i$ for even $i$. We identify a $3$-sphere around $(u,u)$ in $(S^2 \times S^2)_1$ with a $3$-sphere around $(u,v)$ in $(S^2 \times S^2)_2$ via an orientation reversing diffeomorphism which intertwines the representations given by (\ref{equ:rep1}) and (\ref{equ:rep2}). In a similar manner, we identify a $3$-sphere around $(v,u)$ in $(S^2 \times S^2)_2$ with a $3$-sphere around $(v,v)$ in $(S^2 \times S^2)_3$. Continuing in this manner, we obtain the connected sum $\# a(S^2 \times S^2)$ equipped with a $\mathbb{Z}_{2k}$-action. Let $\iota : \# a(S^2 \times S^2) \to \# a(S^2 \times S^2)$ denote the generator of this action.\\

Next, choose a point $x \in \# a (S^2 \times S^2)$ such that the orbit of $x$ under the $\mathbb{Z}_{2k}$-action has size $2k$ (there are infinitely many such orbits). Attach a copy of $-E_8$ at each point of the orbit of $x$. Then we can extend $\iota$ to a $\mathbb{Z}_{2k}$-action on $\# a(S^2 \times S^2) \# 2k(-E_8)$ which cyclically permutes the $2k$ copies of $-E_8$. Repeating this $b$ times we obtain a continuous $\mathbb{Z}_{2k}$-action $\iota$ on $X = \# a (S^2 \times S^2) \# 2kb(-E_8)$. Let $f = \iota_* : H^2(X , \mathbb{Z}) \to H^2(X , \mathbb{Z})$ be the induced action on $H^2(X , \mathbb{Z})$, then (i) holds by construction. The connected sum decomposition of $X$ gives an isomorphism $H^2(X , \mathbb{Z}) = a H \oplus 2kb (-E_8)$. In each copy of $H$, consider the diagonal subspace $H_{\Delta} \subset H$ corresponding to the diagonal copy of $S^2$ in $S^2 \times S^2$. These span a maximal positive definite subspace $H^+ \subset H^2(X , \mathbb{R})$ and it is easily seen that $\iota$ acts as $-Id$ on $H^+$. If $a > 3kb$ then $X$ admits at least one smooth structure since we can view $X$ as $\# (a-3kb)(S^2 \times S^2) \# bk(K3)$. If $a > 3kb$ and $b \ge 1$, then as in the proof of Proposition \ref{p:z2nonrealisation} we can use Wall's theorem \cite[Theorem 2]{wall}, to see that $f$ is realised by some diffeomorphism of $X$, proving (ii). Part (iii) follows from Proposition \ref{p:cyclic}, which applies here since we assume $b^+(X) =a$ and  $k$ are odd and we can take the Spin$^c$-structure $\Gamma$ with $c(\Gamma) = 0$ so that $\Gamma$ is preserved by $\iota$ and $0 = c^2(\Gamma) > \sigma(X) = -16kb$.
\end{proof}

\begin{proposition}
Let $X$ be the topological $4$-manifold $X = \# 3 \mathbb{CP}^2 \# 12 \overline{\mathbb{CP}^2}$. Then there exists an isometry $f : H^2(X , \mathbb{Z}) \to H^2(X , \mathbb{Z})$ of order $4$ which has the following properties: 
\begin{itemize}
\item[(i)]{$f$ acts on an invariant maximal positive definite subspace as the representation $\mathbb{R}_{-} \oplus \mathbb{C}_1$.}
\item[(ii)]{$f$ can be realised by the induced action of a diffeomorphism $X \to X$, where the smooth structure is obtained by viewing $X$ as $\# 3 \mathbb{CP}^2 \# 12 \overline{\mathbb{CP}^2}$.}
\item[(iii)]{$f$ can not be realised by a smooth $\mathbb{Z}_4$-action for any smooth structure on $X$.}
\end{itemize}
\end{proposition}
\begin{proof}
We have $H^2( X , \mathbb{Z}) \cong 3(1) \oplus 12(-1)$. Consider the following elements of $H^2(X , \mathbb{Z})$:
\begin{equation*}
\begin{aligned}
c &= (3 , 1 , 1 ; 1 ,  \dots , 1), \\
x &= (0,2,0 ; 1,1,0, \dots , 0),\\
y &= (3,2,1 ; 1 ,1 , \dots , 1 ), \\
z &= ( 0, 0 , 2 ;  0 , 0 , 1 , 1, 0 , \dots , 0 ).
\end{aligned}
\end{equation*}
Then $c$ is a characteristic, $c^2 = -1 > \sigma(X) = -9$ and $c^2 - \sigma(X) = 8 \; ({\rm mod} \; 16)$. We also find $x^2 = y^2 = z^2 = 2$, $\langle x , y \rangle = 2$, $\langle x , z \rangle = \langle y , z \rangle = \langle x , c \rangle = \langle y , c \rangle = \langle z , c \rangle = 0$. Now let $r_x , r_y , r_z$ be the reflections in the hyperplanes orthogonal to $x,y,z$, so $r_x,r_y,r_z$ are isometries of $H^2(X , \mathbb{Z})$. Let $f = r_x r_y r_z$. Then $f$ preserves the characteristic $c$ and preserves the maximal positive definite subspace $H^+$ given by the span of $x,y,z$. One easily checks that $f |_{H^+}$ has the form $\mathbb{R}_{-} \oplus \mathbb{C}_1$, proving (i). From Proposition \ref{p:cyclic}, we see that $f$ can not be realised by a diffeomorphism of order $4$ for any smooth structure on $X$, proving (iii). For (ii), recall that $\mathbb{CP}^2 \# 2 \overline{\mathbb{CP}^2}$ is diffeomorphic to $(S^2 \times S^2) \# \overline{\mathbb{CP}^2}$ and thus $X$, viewed as the smooth $4$-manifold $\# 3 \mathbb{CP}^2 \# 12 \overline{\mathbb{CP}^2}$ with its standard smooth structure is of the form $(S^2 \times S^2) \# N$ for a smooth compact simply-connected indefinite $4$-manifold $N$ (namely $N = \# 2 \mathbb{CP}^2 \# 11 \overline{\mathbb{CP}^2}$). Thus Wall's theorem \cite[Theorem 2]{wall} applies and we can realise $f$ by a diffeomorphism, proving (ii).
\end{proof}

\subsection{$\mathbb{Z}^d$-actions}

\begin{proposition}
Let $X$ be the topological $4$-manifold $X = \# 2 \mathbb{CP}^2 \# 11 \overline{\mathbb{CP}^2}$. There exists a commuting pair of isometries $f_1 , f_2 : H^2(X , \mathbb{Z}) \to H^2(X , \mathbb{Z})$ such that:
\begin{itemize}
\item[(i)]{$f_1$ and $f_2$ can be realised as diffeomorphisms of $X$ with respect to its standard smooth structure.}
\item[(ii)]{For any smooth structure on $X$, any diffeomorphisms realising $f_1$ and $f_2$ do not commute.}
\end{itemize}
\end{proposition}
\begin{proof}
Let $c,e_1,e_2$ be as in Example \ref{ex:e1e2}. Let $f_1 , f_2$ be the reflections $f_1 = r_{e_1}$, $f_2 = r_{e_2}$. Then $f_1 , f_2$ act on the maximal positive definite subspace $V = span(e_1 , e_2)$ as $diag(-1 , 1)$, $diag(1, -1)$, or $\left( \begin{matrix} \epsilon^1_1 & \epsilon^2_1 \\ \epsilon^1_2 & \epsilon^2_2 \end{matrix} \right) = \left( \begin{matrix} 1 & 0 \\ 0 & 1 \end{matrix} \right)$. It follows from Proposition \ref{p:commutingdiffeos} that $f_1 , f_2$ can't be realised by commuting diffeomorphisms for any smooth structure on $X$, proving (ii). On the other hand, $f_1$ and $f_2$ can each be realised by diffeomorphisms with respect to the standard smooth structure on $X$ by Wall's theorem \cite[Theorem 2]{wall}, proving (i).
\end{proof}

\begin{remark}
A similar type of non-smoothability result for $\mathbb{Z}^2$-actions was obtained for the connected sum of an Enriques surface with $S^2 \times S^2$ by Nakamura \cite{na3}.
\end{remark}

\subsection{$\mathbb{Z}_2 \times \mathbb{Z}_2$-actions}

\begin{proposition}
Let $X$ be the topological $4$-manifold $X = \# 2(a+b) \mathbb{CP}^2 \# (2a+2b+32c+1) \overline{\mathbb{CP}^2}$, where $a,b,c$ are positive integers such that $a,b \ge 3c$. There exists a commuting pair of involutive isometries $\phi_1 , \phi_2 : H^2(X , \mathbb{Z}) \to H^2(X , \mathbb{Z})$ with the following properties:
\begin{itemize}
\item[(i)]{$\phi_1,\phi_2$, can be realised by a continuous, locally linear $\mathbb{Z}_2 \times \mathbb{Z}_2$-action on $X$.}
\item[(ii)]{Viewing $X$ as the smooth $4$-manifold $X = \# (2a+2b-6c) S^2 \times S^2 \# 2c (K3) \# \overline{\mathbb{CP}^2}$, we have that $\phi_1$ and $\phi_2$ can be realised as smooth involutions on $X$.}
\item[(iii)]{For any smooth structure on $X$, we have that $\phi_1$ and $\phi_2$ can not be realised as commuting smooth involutions.}
\end{itemize}
\end{proposition}
\begin{proof}
We first construct $X$ as a topological $4$-manifold with a continuous $\mathbb{Z}_2 \times \mathbb{Z}_2$-action, as illustrated in Figure \ref{fig:3}. We start with $\overline{\mathbb{CP}^2}$ equipped with the following homologically trivial $\mathbb{Z}_2 \times \mathbb{Z}_2$-action:
\[
f_1( [x,y,z]) = [x , -y , z], \quad f_2( [x,y,z] ) = [x,y,-z].
\]
The point $[0,1,0]$ has orbit $\{ [0,1,0] , [0,-1,0] \}$ and stabiliser $\mathbb{Z}_2 = \langle f_2 \rangle$. The action of $f_2$ on the tangent space at $[0,1,0]$ has the form ${\rm diag}(1,1,-,1,-1)$. Similarly, the point $[0,0,1]$ has orbit $\{ [0,0,1] , [0,0,-1] \}$ and stabiliser $\mathbb{Z}_2 = \langle f_1 \rangle$. The action of $f_1$ on the tangent space at $[0,0,1]$ has the form ${\rm diag}(1,1,-1,-1)$. Now consider $S^2 \times S^2$ equipped with the $\mathbb{Z}_2$-action $\iota(x,x) = (r(x) , r(x))$, where $r : S^2 \to S^2$ is a reflection. Clearly we can form a $\mathbb{Z}_2$-equivariant connected sum $\# a(S^2 \times S^2)$ by attaching one copy of $S^2 \times S^2$ to the next one at fixed points. Next, we connect sum $\# a(S^2 \times S^2)$ to $\overline{\mathbb{CP}^2}$ at $[0,0,1]$ in a $\mathbb{Z}_2 = \langle f_1 \rangle$-equivariant fashion and similarly attach $\# a(S^2 \times S^2)$ to $[0,0,-1]$ $f_1$-equivariantly. Since $f_2$ exchanges the points $[0,0,1]$ and $[0,0,-1]$, it is clear that we now obtain a $\mathbb{Z}_2 \times \mathbb{Z}_2$-action on $\# 2a(S^2 \times S^2) \# \overline{\mathbb{CP}^2}$, where $f_2$ exchanges the two copies of $\# a(S^2 \times S^2)$. Similarly, we connect sum $\# b(S^2 \times S^2)$ to the point $[0,1,0]$ in a $\mathbb{Z}_2 = \langle f_1 \rangle$-equivariant manner and also attach $\# b(S^2 \times S^2)$ to $[0,-1,0]$ $f_1$-equivariantly. In this way, we have obtained a $\mathbb{Z}_2 \times \mathbb{Z}_2$-action on $X' = \#(2a+2b)S^2 \times S^2 \# \overline{\mathbb{CP}^2}$, where $f_2$ exchanges the two copies of $\#a(S^2 \times S^2)$ and $f_1$ exchanges the two copies of $\# b(S^2 \times S^2)$.\\

Let $x$ be any point in $X'$ in which the orbit has size $4$ and attach a copy of $-E_8$ to each of these four points. Repeating this $c$ times, we get a continuous $\mathbb{Z}_2 \times \mathbb{Z}_2$-action on $X = X' \# 4c(-E_8)$. If $a,b \ge 3c$ then $X$ is clearly smoothable, so from Freedman's classification of compact simply-connected topological $4$-manifolds, we have $X = \# 2(a+b) \mathbb{CP}^2 \# (2a+2b+32c+1) \overline{\mathbb{CP}^2}$. Let $\phi_1 , \phi_2 : H^2(X  \mathbb{Z}) \to H^2(X , \mathbb{Z})$ be the isometries of $H^2(X , \mathbb{Z})$ induced by the involutions $f_1 , f_2 : X \to X$. Then (i) holds by construction. Let $c = 3 \in \mathbb{Z} = H^2( \overline{\mathbb{CP}^2} , \mathbb{Z})$. Then we can regard $c$ as an element of $H^2( X , \mathbb{Z})$ and we have that $c$ is a characteristic and that $c^2 = -9$. It follows that $c^2 - \sigma(X) = 8 \; ({\rm mod} \; 16)$ and that $c^2 - \sigma(X) > 0$ as long as $c \ge 1$.\\

\begin{figure}

\begin{tikzpicture}

\draw[thick] (0,0) circle (1);
\node at (0,0) {$\#^a S^2 \times S$};

\draw[thick] (2.5,0) circle (1);
\node at (2.5,0) {$\overline{\mathbb{CP}^2}$};

\draw[thick] (5,0) circle (1);
\node at (5,0) {$\#^a S^2 \times S^2$};

\draw[thick] (0,2.5) circle (1);
\node at (0,2.5) {$\#^c (-E_8)$};

\draw[thick] (2.5,2.5) circle (1);
\node at (2.5,2.5) {$\#^b S^2 \times S^2$};

\draw[thick] (5,2.5) circle (1);
\node at (5,2.5) {$\#^c (-E_8)$};

\draw[thick] (0,-2.5) circle (1);
\node at (0,-2.5) {$\#^c (-E_8)$};

\draw[thick] (2.5,-2.5) circle (1);
\node at (2.5,-2.5) {$\#^b S^2 \times S^2$};

\draw[thick] (5,-2.5) circle (1);
\node at (5,-2.5) {$\#^c (-E_8)$};

\draw[thick](1,0)--(1.5,0);
\draw[thick](3.5,0)--(4,0);

\draw[thick](1,2.5)--(1.5,2.5);
\draw[thick](3.5,2.5)--(4,2.5);
\draw[thick](1,-2.5)--(1.5,-2.5);
\draw[thick](3.5,-2.5)--(4,-2.5);

\draw[thick](2.5,1)--(2.5,1.5);
\draw[thick](2.5,-1)--(2.5,-1.5);

\draw[thick, <->]
(6.5 , -2.35) arc (-70:70: 1 and 2.5);

\node at (7.5,0) {{\Large $f_1$}};

\draw[thick, <->]
(0.1, -3.8) arc (-160:-20: 2.5 and 1);

\node at (2.5,-4.8) {{\Large $f_2$}};

\end{tikzpicture}

\caption{The $4$-manifold $\# 2(a+b) \mathbb{CP}^2 \# (2a+2b+32c+1) \overline{\mathbb{CP}^2}$ with $\mathbb{Z}_{2} \times \mathbb{Z}_2$-action.}\label{fig:3}
\end{figure}
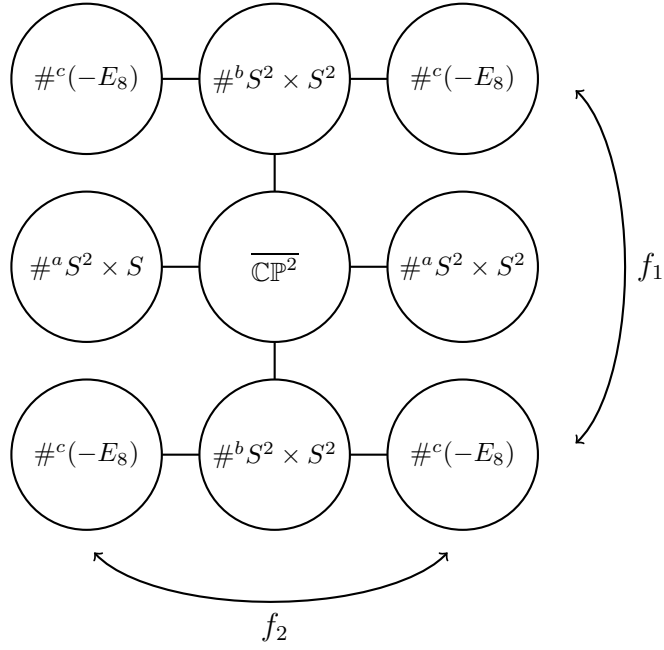

For each copy of $(S^2 \times S^2)$ we attached to $\overline{\mathbb{CP}^2}$, consider the $1$-dimensional subspace of $H^2(X , \mathbb{R})$ spanned by the Poinca\'re dual of the diagonal embedding $S^2 \to S^2 \times S^2$. The direct sum $V$ of these $1$-dimensional spaces is a maximal positive definite subspace of $H^2(X , \mathbb{R})$ preserved by $\phi_1$ and $\phi_2$. A simple calculation shows that (in the notation used in the proof of Proposition \ref{p:z2z2}):
\[
V = \mathbb{R}_{-1 , 1}^a \oplus \mathbb{R}_{1,-1}^b \oplus \mathbb{R}_{-1,-1}^{a+b}.
\]
In particular, there are no $\mathbb{R}_{1,1}$ summands. Therefore, Proposition \ref{p:z2z2} implies that the $\mathbb{Z}_2 \times \mathbb{Z}_2$-action on $H^2(X , \mathbb{Z})$ generated by $\phi_1 , \phi_2$ can not be realised by a smooth $\mathbb{Z}_2 \times \mathbb{Z}_2$-action, for any smooth structure on $X$, proving (iii).\\

To complete the proof it remains to show (ii). Let us view $X$ as the smooth $4$-manifold $X = \# (2a+2b-6c) S^2 \times S^2 \# 2c (K3) \# \overline{\mathbb{CP}^2}$. We will show that $\phi_1$ can be realised by a smooth involution. A similar argument applies for $\phi_2$. First, start with $\overline{\mathbb{CP}^2}$ with the same action of $f_1$ as before, i.e. $f_1( [x,y,z]) = [x,-y,z]$. As before, attach $\# a(S^2 \times S^2)$ to $[0,0,1]$ and $\# a(S^2 \times S^2)$ to $[0,0,-1]$ $f_1$-equivariantly. Also attach $\# (b-3c)(S^2 \times S^2)$ to $[0,1,0]$ and $\# (b-3c) (S^2 \times S^2)$ to $[0,-1,0]$ so that $f_1$ exchanges these two copies of $\# (b-3c)(S^2 \times S^2)$. Lastly, $f_1$-equivariantly attach $c$ pairs of copies of $K3$, to obtain a smooth $\langle f_1 \rangle = \mathbb{Z}_2$-action on $\# (2a+2b-6c) S^2 \times S^2 \# 2c (K3) \# \overline{\mathbb{CP}^2} = X$. It is not hard to see that this construction is such that the induced action of this involution on $H^2(X , \mathbb{Z})$ agrees with $\phi_1$.
\end{proof}

\begin{remark}
A similar type of non-smoothability result for $\mathbb{Z}_2 \times \mathbb{Z}_2$-actions on $X = \#(2l_1 + 2l_2 + 1 - 6k)S^2 \times S^2 \# 2k (K3)$, $l_1,l_2 \ge 3k$, $k \ge 1$ was obtained in \cite{kato}. It is interesting to note that \cite{kato} does not use families Seiberg-Witten theory and instead obtains the result using an equivariant version of Furuta's $10/8$-inequality \cite{fur}.
\end{remark}


\bibliographystyle{amsplain}

\begin{thebibliography}{99}

\bibitem{fre}M. H. Freedman, The topology of four-dimensional manifolds, {\em J. Differential Geom.} {\bf 17} (1982), no. 3, 357-453. 
\bibitem{frmo}R. Friedman, J. W. Morgan, On the diffeomorphism types of certain algebraic surfaces. I, {\em J. Differential Geom.} {\bf 27} (1988), no. 2, 297-369. 
\bibitem{fur}M. Furuta, Monopole equation and the $\frac{11}{8}$-conjecture, {\em Math. Res. Lett.} {\bf 8} (2001), no. 3, 279-291. 
\bibitem{kato}Y. Kato, Nonsmoothable actions of $\mathbb{Z}_2 \times \mathbb{Z}_2$ on spin four-manifolds, arXiv:1708.08030 (2017).
\bibitem{ker}S. P. Kerckhoff, The Nielsen realization problem, {\em Ann. of Math.} (2) {\bf 117} (1983), no. 2, 235-265. 
\bibitem{konno}H. Konno, Bounds on genus and configurations of embedded surfaces in 4-manifolds, {\em J. Topol.} {\bf 9} (2016), no. 4, 1130-1152. 
\bibitem{konno2}H. Konno, Characteristic classes via $4$-dimensional gauge theory, arXiv:1803.09833 (2018).
\bibitem{kre}M. Kreck, {\em Isotopy classes of diffeomorphisms of $(k-1)$-connected almost-parallelizable $2k$-manifolds}. Algebraic topology, Aarhus 1978 (Proc. Sympos., Univ. Aarhus, Aarhus, 1978), pp. 643-663, Lecture Notes in Math., {\bf 763}, Springer, Berlin, 1979. 
\bibitem{lili}B.-H. Li, T.-J. Li, Symplectic genus, minimal genus and diffeomorphisms, {\em Asian J. Math.} {\bf 6} (2002), no. 1, 123-144. 
\bibitem{liliu}T.-J. Li, A.-K. Liu, Family Seiberg-Witten invariants and wall crossing formulas, {\em Comm. Anal. Geom.} {\bf 9} (2001), no. 4, 777-823. 
\bibitem{lon}M. L\"onne, On the diffeomorphism groups of elliptic surfaces, {\em Math. Ann.} {\bf 310} (1998), no. 1, 103-117. 
\bibitem{mor}S. Morita, Characteristic classes of surface bundles, {\em Invent. Math.} {\bf 90} (1987), no. 3, 551-577. 
\bibitem{na1}N. Nakamura, The Seiberg-Witten equations for families and diffeomorphisms of 4-manifolds, {\em Asian J. Math.} {\bf 7} (2003), no. 1, 133-138. 
\bibitem{na2}N. Nakamura, Correction to: ``The Seiberg-Witten equations for families and diffeomorphisms of 4-manifolds", Asian J. Math., Vol. 7, No. 1, 133--138, 2003, {\em Asian J. Math.} {\bf 9} (2005), no. 2, 185-186.
\bibitem{na3}N. Nakamura, Smoothability of $\mathbb{Z} \times \mathbb{Z}$-actions on 4-manifolds, {\em Proc. Amer. Math. Soc.} {\bf 138} (2010), no. 8, 2973-2978. 
\bibitem{nic}L. I. Nicolaescu, {\em Notes on Seiberg-Witten theory}. Graduate Studies in Mathematics, 28. American Mathematical Society, Providence, RI, 2000. xviii+484 pp.
\bibitem{qui}F. Quinn, Isotopy of $4$-manifolds, {\em J. Differential Geom.} {\bf 24} (1986), no. 3, 343-372. 
\bibitem{rub}D. Ruberman, Positive scalar curvature, diffeomorphisms and the Seiberg-Witten invariants, {\em Geom. Topol.} {\bf 5} (2001), 895-924. 
\bibitem{sal}D. Salamon, Spin geometry and Seiberg-Witten invariants, unpublished notes.
\bibitem{wall}C. T. C. Wall, Diffeomorphisms of $4$-manifolds, {\em J. London Math. Soc.} {\bf 39} 1964 131-140. 

\end{thebibliography}

\end{document}